\documentclass{amsart}

\usepackage{amsrefs}
\usepackage{hyperref}

\usepackage{amsmath,amsthm,color}
\usepackage {latexsym}
\usepackage{amssymb}

\newcommand{\beal}{\begin{align}}
\newcommand{\enal}{\end{align}}
\newcommand{\bealn}{\begin{align*}}
\newcommand{\enaln}{\end{align*}}
\newcommand{\bear}{\begin{eqnarray}}
\newcommand{\eear}{\end{eqnarray}}
\newcommand{\beeq}{\begin{equation}}
\newcommand{\eneq}{\end{equation}}

\newcommand{\eps}{{\varepsilon}}
\newcommand{\R}{{\mathbb R}}

\newcommand{\Z}{{\mathbb Z}}

\newcommand{\la}{\langle}
\newcommand{\ra}{\rangle}

\def\bm{\left[ \begin{array}{cc}}
\def\endm{\end{array}\right]}

\def\eps{\varepsilon}

\def\bm{\left[\begin{matrix} }
\def\endm{\end{matrix}\right]}
\def\la{\langle}
\def\ra{\rangle}

\def\R{{\mathbb R}}

\newtheorem{theorem}{Theorem}
\newtheorem{lemma}[theorem]{Lemma}

\newtheorem{prop}[theorem]{Proposition}

\theoremstyle{remark}
\newtheorem{remark}[theorem]{Remark}

\renewcommand{\Re}{\,{\rm Re}\,}

\renewcommand{\hat}{\widehat}
\renewcommand{\epsilon}{\eps}
\renewcommand{\tilde}{\widetilde}
\numberwithin{equation}{section}
\numberwithin{theorem}{section}

\newcommand{\eqdef}{\overset{\mbox{\tiny{def}}}{=}}

\newcommand{\dNOTj}{D^{(j)}}
\newcommand{\dNOTjO}{D^{(j-1)}}
\newcommand{\dNOTjT}{D^{(j-2)}}
\newcommand{\ang}[1]{ \left< {#1} \right> }

\renewcommand{\dim}{3}
\newcommand{\threed}{{\mathbb R}^\dim}

\begin{document}


\title[Global solutions to a non-local diffusion equation]{Global solutions to a non-local diffusion equation with quadratic non-linearity}

\author[J. Krieger]{Joachim Krieger}
\address{(JK) Department of Mathematics, The University of Pennsylvania, 209 South 33rd Street, Philadelphia, PA 19104, U.S.A.}
\email{kriegerj at math.upenn.edu}
\urladdr{http://www.math.upenn.edu/~kriegerj/}
\thanks{J.K. was partially supported by the NSF grant DMS-0757278.}

\author[R. M. Strain]{Robert M. Strain}
\address{(RMS) 
University of Pennsylvania, Department of Mathematics, David
Rittenhouse Lab, 209 South 33rd Street, Philadelphia, PA 19104, U.S.A.}
\email{strain at math.upenn.edu}
\urladdr{http://www.math.upenn.edu/~strain/}
\thanks{R.M.S. was partially supported by the NSF grant DMS-0901463, and an Alfred P. Sloan Foundation Research Fellowship.}


\setcounter{tocdepth}{1}

\begin{abstract}
In this paper we prove the global in time well-posedness of the following non-local diffusion equation with $\alpha \in (0,2/3)$:
$$
\partial_t u = \left\{(-\triangle)^{-1}u\right\} \triangle u + \alpha u^2,
\quad  u(t=0) = u_0.
$$
The initial condition $u_0$ is positive, radial, and non-increasing with
$u_0\in L^1\cap L^{2+\delta}(\threed)$ for some small $\delta >0$.  There is no size restriction on $u_0$.
This model problem appears of interest due to its structural similarity with Landau's equation from plasma physics, and moreover its radically different behavior from the semi-linear Heat equation: $u_t = \triangle u + \alpha u^2$. 
\end{abstract}


\maketitle
\tableofcontents

\thispagestyle{empty}

\section{Introduction and main results}

We study the following model equation for $\alpha \in (0, 2/3)$:
\begin{equation}\label{eqn:model}
\partial_t u = \left\{(-\triangle)^{-1}u\right\} \triangle u + \alpha u^2,\quad  u(0, x) = u_0,
\end{equation}
where as usual
$$
(-\triangle)^{-1}u 
=
\left(\frac{1}{4\pi |\cdot |} * u\right)(x)
=
\frac{1}{4\pi}\int_{\threed}dy~ \frac{u(y)}{|x-y|}.
$$  
We also consider $(t, x)\in \R_{\geq 0}\times\R^{3}$. Moreover, we shall restrict to $u_0$ positive and radial; a condition which is propagated by the equation.  Note
$$
\int_{\R^3} dx ~ u(t,x) + (1-\alpha) \int_0^t\int_{\R^3} dsdx ~ |u(s,x)|^2 = \int_{\R^3} dx ~ u_0(x).
$$
In other words for solutions to \eqref{eqn:model}, the quantity above is formally conserved.

Our motivation is partially derived from the spatially-homogeneous {\it{Landau equation 1936}} \cite{MR684990} in plasma physics, which takes the form 
$$
\partial_t f = \mathcal{Q}(f,f),
$$
where for $\partial_i = \frac{\partial}{\partial v_i}$ we have
$$
\mathcal{Q}(f,f)
\eqdef 
\sum_{i,j=1}^\dim \partial_i \int_{\threed} dv_* ~ a^{ij}(v-v_*) \left\{f(v_*) (\partial_j f)(v)  - f(v) (\partial_j f)(v_*)  \right\}.
$$
Here the projection matrix is given by
$$
a^{ij}(v) = \frac{L}{8\pi } |v|^{\gamma + 2} \left( \delta_{ij} - \frac{v_i v_j}{|v|^2} \right), \quad L > 0.
$$
The parameter  satisfies $\gamma \ge -3$, and we are solely concerned with the main physically relevant Coulombian case of $\gamma = -3$.  Then formally differentiating under the integral sign and integrating by parts we obtain 
$$
\mathcal{Q}(f,f) = \sum_{i,j=1}^\dim \bar{a}_{ij}(f)\partial_{i} \partial_{j} f  - 
\left(\int_{\threed} dv_* \sum_{i,j=1}^\dim \partial_i \partial_j a^{ij}(v-v_*) f(v_*) \right) f(v),
$$
where
\[
\bar{a}_{ij}(f) \eqdef \left(\frac{L}{8\pi | v|}\left(\delta_{ij}-\frac{v_i v_j}{| v|^2}\right)\right) * f.
\]
Furthermore $\sum_{i,j=1}^\dim \partial_i \partial_j a^{ij}(v-v_*)$  is a delta function, so that
\begin{equation}\label{Landau}
\partial_t f = \sum_{1\leq i, j\leq 3}\bar{a}_{ij}(f)\partial_{i} \partial_{j} f+ L f^2,\quad (t, v)\in \R_{\geq 0}\times\threed.  
\end{equation}
See \cite[Page 170, Eq. (257)]{MR1942465}.  We can set $L=1$ for simplicity.  

It is well known that non-negative solutions to \eqref{Landau} preserve the $L^1$ mass.  This suggests that 
\eqref{eqn:model} with $\alpha=1$ may be a good model for solutions to the Landau equation \eqref{Landau}. 
In particular we consider it important that these two models, \eqref{eqn:model} and \eqref{Landau},  have the same quadratic non-linearity and the same ``Coulomb'' type singularity multiplying in front of the diffusion (although the Landau equation also can be anisotropic).  
It appears that neither existence of global strong solutions for general large data, nor formation of singularities is known for either \eqref{Landau}, or \eqref{eqn:model}. In fact, we are not aware of any earlier studies of \eqref{eqn:model} in the literature.

For the Landau equation \eqref{Landau}, Desvillettes and Villani  \cite{MR1737547} have established the global existence of unique weak solutions and the instantaneous smoothing effect for a large class of initial data in the year 2000 with $\gamma \ge 0$.  Then Guo \cite{MR1946444} in 2002 proved the existence of classical solutions with the physical Coulombian interactions ($\gamma = -3$) for smooth nearby Maxwellian initial data.  For further  results in these directions we refer to \cite{MR2506070,MR2502525,MR1055522,MR1942465,MR2100057} and the references therein.

Furthermore, it is well known that the nonlinear heat equations such as
$$
\partial_t u = \Delta u + \alpha u^2, \quad \alpha > 0,
$$
will experience blow-up in finite time even for small initial data.  This problem has a long and detailed history which we omit.  We however refer to the results and discussion in \cite{MR1488298}, and the references therein, for more on this topic.

At one point, \cite[Page 170, Eq. (257)]{MR1942465}, it was thought that equations such as \eqref{Landau} could generally blow up in finite time.  It was a common point of view that the diffusive effects of the Laplace operator would be too weak to prevent the blow-up effects that are caused by a quadratic source term.  Then since the diffusion matrix such as $\bar{a}_{ij}(f)$ or $(-\triangle)^{-1}u$ may be bounded (or decay at infinity, such as in \eqref{asymptoticU} and \cite{MR1946444}) then blow-up may indeed occur, as is the case for the Heat equation.  

This intuition may no longer be as widespread as it once was for the Landau equation \cite{MR1942465}, in particular because it has a divergence structure and since also there seems to be lack of numerical simulations finding blow-up.  Yet these issues have still been without rigorous clarification.

Furthermore, for $u$ non-negative, we have that 
\[
\bar{a}_{ij}(u)\leq (-\triangle)^{-1}u.
\]
This gives the expectation that the diffusive 
effects of $ (-\triangle)^{-1}u \triangle u$ will be stronger than those of $ \sum_{1\leq i, j\leq 3}\bar{a}_{ij}(f)\partial^2_{ij}f$. 

The main contribution of this paper is to show that in contrast to the behavior of nonlinear heat equations, solutions to \eqref{eqn:model} indeed can exist globally in time even for large radial monotonic initial data.  We initiate the study of \eqref{eqn:model}, and attempt to construct global solutions for $\alpha>0$ as close to $1$ as possible. We  have 
\begin{theorem}\label{thm:Main} Let $0\leq\alpha<\frac{1}{2}$.  Suppose that $u_0(x)$ is positive, radial, and non-increasing with $u_0\in L^1(\threed)\cap L^{2+}(\threed)$.  Additionally suppose that  $-\triangle \tilde{u}_0\in L^2(\threed)$, where $\tilde{u}_0 \eqdef \la x \ra^{\frac{1}{2}}u_0$. 
Then there exists a unique global solution with\footnote{By this we mean that there is only one solution satisfying all these conditions.} 
\[
u(t, x)\in C^0([0, \infty), L^1\cap L^{2+}(\threed))\cap C^0({\R}_{\geq 0}, H^2(\threed)),
\]
\[
\la x \ra^{\frac{1}{2}}(-\triangle)u(t, x)\in C^0([0, \infty), L^2(\threed)).
\]
The solution decays toward zero at $t=+\infty$, in the following sense: 
\[
\lim_{t\to\infty}\| u(t, \cdot)\|_{L^q(\threed)}=0, \quad q\in (1, 2].
\]
\end{theorem}

Above the space $L^{2+}$ means that there exists a small $\delta >0$ such that we are in the space $L^{2+\delta}(\threed)$.  
Furthermore we use the notation 
$\la x \ra \eqdef \sqrt{ 1+ |x|^2 }$.  Also the space $X$ is defined by $X\eqdef L^1\cap L^{2+}(\threed)$.
\begin{remark} Due to instantaneous smoothing for parabolic equations, one may expect to strengthen the above result to the effect that $u\in C^\infty(\R_{+}\times \threed)$. 
\end{remark}

The reason for the upper bound $\alpha<\frac{1}{2}$ comes from the interplay of the  local well-posedness
 we can establish for \eqref{eqn:model}, and global a priori bounds. In effect, we shall show that this problem is strongly locally well-posed for data of the form of the theorem. Furthermore, the equation immediately implies a priori bounds for the norms $\| u(t, \cdot)\|_{L^q(\threed)}$ for $1\leq q\leq 2+\delta$ for a small $\delta = \delta(\alpha)>0$. The quasilinear character imposes the added difficulty of establishing the non-degeneration of the operator $\big\{(-\triangle)^{-1}u\big\}\triangle$, which we ensure by exploiting the additional symmetries/monotonicity properties of the data. Note that the method employed in this paper suggests the threshold of $\alpha\leq \frac{2}{3}$ which corresponds to conservation of  $L^{\frac{3}{2}}$-norm. More precisely, we are able to show that under the restriction $\alpha<\frac{2}{3}$ we can deduce a priori bounds on the $L^{2+}$-norm on finite tine intervals. This allows us to strengthen the preceding theorem by exploiting a more subtle a priori bound to get

\begin{theorem}\label{thm:MainImprov} Let $0\leq\alpha<\frac{2}{3}$, and $u_0$ be as in Theorem~\ref{thm:Main}. Then there exists a global solution in the same spaces as in Theorem~\ref{thm:Main}; this solution  further satisfies
\[
\lim_{t\to\infty}\|u(t, \cdot)\|_{L^q(\threed)}=0,\quad q\in (1,3/2].
\]
\end{theorem}  

We do not think it unreasonable to conjecture actual strong local well-posedness for data in $L^{\frac{3}{2}+}$. This appears as a natural limit for the well-posedness of \eqref{eqn:model} in light of the optimal local well-posedness\footnote{This equation is strongly locally well-posed in $L^{\frac{3}{2}+}$,  \cite{MR1658660}.} for 
\[
\partial_t u = \triangle u + u^2,
\]
established in \cite{MR1658660}. The preceding theorem then appears as a natural endpoint to the method employed in this paper, although it may of course still be possible to deduce a priori bounds that allow one to get above $\alpha<\frac{2}{3}$ for global existence.
\\

Indeed, it  appears that the case  $\alpha = 1$ is the natural threshold for global well-posedness.  It further seems reasonable to conjecture that increasing $\alpha$  beyond $\alpha>1$, one should get finite time blow up solutions. 
We  are unable to show this, but we do have the following simple example:

\begin{prop}\label{prop: boundedblowup} Consider \eqref{eqn:model} but on the ball $B_1(0)\eqdef \{x\in \R^3 ~ | ~ |x|\leq 1\}$, choosing $(-\triangle)^{-1}u$ to have vanishing values on $\{|x| = 1\}$, $\alpha>1$. Then nontrivial non-negative smooth global solutions of \eqref{eqn:model} vanishing on $\partial B_1(0)$ cannot exist. 
\end{prop}
\begin{proof} Let $u(t, x)\geq 0$ be such a solution, $t\geq 0$. Then using integration by parts:
\[
\frac{d}{dt}\int_{B_1(0)}u(t, x)dx = (\alpha-1)\int_{B_1(0)}u^2(t, x)dx\gtrsim \left(\int_{B_1(0)}u(t, x)dx\right)^2, 
\]
using the H{\"o}lder inequality in the last step. But then we infer 
\[
\lim_{t\to T}\int_{B_1(0)}u(t, x)dx = \infty,
\]
 for some $T>0$ since $\int_{B_1(0)}u_0 dx >0$. 
\end{proof}

The difficulty in extending this reasoning to the context of $\threed$ is that the $L^1$-mass could spread out to spatial infinity `too quickly'. 

We will use the notation $A \lesssim B$ to mean that there exists an inessential uniform constant $C>0$ such that $A \le C B$. In general $C$ will denote an inessential uniform constant whose value may change from line to line.  Furthermore,
$A \gtrsim B$ means $B \lesssim A$, and $A \approx B$ is defined as $A \lesssim B \lesssim A$.

In the next section we discuss the local existence theory.  Then in Section \ref{sec:global} we extend this local existence theory globally in time and prove the decay rates as $t \to \infty$; both of these make use of a monotonicity formula. More precisely, we shall show that the expressions $\|u(t, \cdot)\|_{L_x^q}$ are non-increasing or at least locally bounded in time for a suitable range of $q$. The local existence theory then gives the global existence. 

\section{Local existence theory}\label{sec:local}

Our main result in this section is the following local existence theorem:

\begin{prop}\label{prop:local} Consider \eqref{eqn:model} with $\alpha \ge 0$, and let $u_0$, $\tilde{u}_0$ be as in Theorem~\ref{thm:Main}. Pick $1>r_0>0$ such that 
\begin{equation}
\int_{r_0^{-1}>|x|>r_0}u_0(x)dx >0. 
\label{initialNON}
\end{equation}
Then there exists 
\[
T = T\left(\|u_0\|_{L^1\cap L^{2+}(\threed)}+\|\triangle \tilde{u}_0\|_{L^2(\threed)}, ~r_0, ~\int_{r_0^{-1}>|x|>r_0}u_0(x)dx\right)>0,
\]
and a unique solution $u(t, x)$ on $[0, T)\times\threed$ satisfying the following properties\footnote{By this we mean that there exists a unique solution in the class of solutions satisfying all these requirements.}: 
 $u(t, x)$ is radial, non-increasing, and positive.  Furthermore
 $$
 u\in C^{0}([0, T), L^1\cap L^{2+}(\threed)), 
 \quad \la x \ra^{\frac{1}{2}}\triangle u\in L^2(\threed).
 $$
 Finally we have the pointwise bound 
\begin{equation}
 D_1>(-\triangle)^{-1}u(t, x)>\frac{D_2}{\la x \ra},\quad D_{1},   D_{2}>0.
 \label{boundUinv}
\end{equation}
This holds uniformly on $[0, T)\times\threed$.
\end{prop}

\subsection{Some useful facts}
We recall the Newton formula 
for radial functions (see the book of Lieb and Loss \cite[Theorem 9.7]{MR1817225}):
\begin{equation}\label{NewtonFormula3D}
\begin{split}
(-\triangle)^{-1}u(x) 
& = \frac{1}{4\pi |x|}\int_{|y|\leq |x|}u(y)dy + \int_{|y|\geq |x|}\frac{u(y)}{4\pi |y|}dy,
\\
& = \frac{1}{3 |x|}\int_{0}^{|x|} u(\rho) \rho^{2} d\rho + \frac{1}{3 } \int_{|x|}^\infty u(\rho)  \rho d\rho.
\end{split}\end{equation}
We claim that \eqref{NewtonFormula3D} combined with\footnote{We use the notation $\|u\|_{L_t^p L_x^q}: = \big\|\|u(t,\cdot)\|_{L_x^q}\big\|_{L_t^p}$.} $u\in L^\infty_t L^1_x$ implies
$$
(-\triangle)^{-1}u(t,x)\leq \frac{\tilde{D}_1}{\la x \ra},
\quad 
\tilde{D}_1 >0.
$$
This follows easily by splitting into the separate regions $|x| \ge 1$ and $|x| \le 1$.  On the former region we use 
Newton's formula \eqref{NewtonFormula3D}
and on the latter region we use 
the upper bound in \eqref{boundUinv}.  Combining this with \eqref{boundUinv}, we conclude that we have uniformly on $[0,T)\times \threed$ 
\begin{equation}
(-\triangle)^{-1}u(t,x)  \approx  \la x \ra^{-1}.
\label{asymptoticU}
\end{equation}
This estimate will be used several times below. 

We shall also use {\it{Littlewood-Paley frequency cutoffs}} $P_{<a}$, $P_b$, where $a$, $b$ are integers. To define these, one uses a suitable radial smooth bump function $m(\xi)\in C_c^\infty(\R^3)$ with $\sum_{k\in\Z}m(2^k\xi) = 1$, $\xi\in \R^3\backslash \{0\}$, and then puts 
\[
\widehat{P_k f}(\xi) = m(\frac{\xi}{2^k})\hat{f}(\xi),\,\widehat{P_{<a} f}(\xi) = \sum_{k<a}m(\frac{\xi}{2^k})\hat{f}(\xi)
\]
and analogously for $P_{>a}$. Then one has the useful {\it{Bernstein's inequality}}
\[
\|P_{<l}f\|_{L^q(\R^3)}\lesssim 2^{\frac{3l}{p} - \frac{3l}{q}}\|P_{<l}f\|_{L^p(\R^3)},\quad1\leq p\leq q\leq \infty.
\]

\subsection{The iteration scheme}
We prove Proposition \ref{prop:local} by constructing a local solution by means of an iteration scheme. Specifically we set
$$
u^{(0)}(t, x) \eqdef e^{t\triangle}u_0(x), \quad t\in [0, T), 
$$
and then we define implicitly 
\begin{equation}\label{eqn:modelj}
\begin{split}
\partial_t u^{(j)}(t, x) &=  (-\triangle)^{-1}(u^{(j-1)})\triangle u^{(j)} + \alpha \big(u^{(j-1)}\big)^2,\quad 
j\in\{ 1,2,\ldots\},
\\
u^{(j)}(0, x) &= u_0(x).
\end{split}
\end{equation}
Our goal will be to establish the uniform estimates in the following lemma: 

\begin{lemma}\label{lem:bounds}(Key Lemma)
There exists $T>0$ as well as $D_{i}>0$ $(i=1,2,3)$, all depending on $r_0$, \eqref{initialNON} and 
$ \|u_0\|_{L^1\cap L^{2+}(\threed)}$ such that we have the following uniform bound  $\forall j\geq 0$: 
\begin{equation}\label{lem:bound1}
\|u^{(j)}\|_{L_t^{\infty}\left([0,T);L^1\cap L^{2+}(\threed) \right)}+\sup_{t\in (0, T]}t^{\frac{1}{2}}\|\chi_{|x|\lesssim 1}u^{(j)}\|_{L^6(\threed)}<D_3,
\end{equation}
where $\chi_{|x|\lesssim 1}$ smoothly truncates to the indicated region $(|x|\lesssim 1)$, and such that 
\eqref{boundUinv} and \eqref{asymptoticU} hold for $u = u^{(j)}$ $\forall j\geq 0$ uniformly.  
Moreover all the $u^{(j)}(t, \cdot)$ are non-increasing, positive, radial and we obtain the uniform derivative bounds 
\begin{equation}\label{lem:bound2}
\|\la x \ra^{\frac{1}{2}}\nabla^\alpha u^{(j)}(t,\cdot)\|_{L^2(\threed)}\leq D_4,\quad 0\leq |\alpha|\leq 2,
\end{equation}
where $D_4$ depends on the same quantities as $D_{i}$ $(i=1,2,3)$ and it additionally depends linearly on $\|\triangle \tilde{u}_0\|_{L^2}$. 
\end{lemma}

The proof of Lemma \ref{lem:bounds} is the core of the paper and extends up to Section \ref{sec:higherDb}.  We proceed by induction on $j$. In the case $j=0$, the bounds 
\[
\|e^{t\triangle}u_0\|_{L_t^{\infty}\left([0,T);L^1\cap L^{2+}(\threed) \right)}\leq \|u_0\|_{L^1\cap L^{2+}(\threed)},
\]
follow from the explicit form of the heat kernel.  
Further the bound 
\[
\|e^{t\triangle}u_0\|_{L^{6}(\threed)}\lesssim t^{-\frac{1}{2}}\|u_0\|_{L^2(\threed)},
\]
follows from the Sobolev embedding after applying $\nabla$. Also, clearly $u^{(0)}$ will be radial and positive throughout, as well as non-increasing.  Furthermore the formula \eqref{NewtonFormula3D} combined with a simple continuity argument as well as the H{\"o}lder inequality allow us to conclude that \eqref{boundUinv} holds for 
$
u^{(0)}(t, x),
$
where the constants $D_{i}$ depend upon $\|u_0\|_{L^1\cap L^{2+}(\threed)}$,  $r_0$ and $\int_{r_0^{-1}>|y|>r_0}u_0(y)dy>0$
for 
 $i=1,2$. 
 
 Indeed, to obtain the lower bound, we use 
 \begin{multline}\label{eqn: conttrick}
 \frac{1}{4\pi|x|}\int_{|y|\leq |x|}u^{(0)}(t, y)\,dy 
 + 
 \int_{|y|\geq|x|}\frac{u^{(0)}(t, y)}{4\pi|y|}\,dy
 \\
 \geq 
 \frac{r_0}{8\pi\la x\ra}\int_{|y|<2r_0^{-1}}u^{(0)}(t, y)\,dy.
 \end{multline}
This is clear when $|x|\ge \frac{2}{r_0}$, in which case already the first integral is bounded from below by the right hand side; on the other hand, when $|x|< \frac{2}{r_0}$, we minorize the second integral by $\frac{r_0}{8\pi}\int_{|x|\leq |y|\leq \frac{2}{r_0}}u^{(0)}(t, y)\,dy$ and exploit the fact that $r_0<1$.
 Then if $\chi$ is a non-negative smooth cutoff which equals $1$ on $\{|y|<r_0^{-1}\}$ and has compact support on $|y|<2r_0^{-1}$, we have 
 \[
\big|\frac{d}{dt}\big(\int \chi(y) u^{(0)}(t, y)\,dy\big)\big|\lesssim \int u^{(0)}(t, y)\,dy = \int u_0\,dy,
\]
whence we obtain $\int_{|y|<2r_0^{-1}}u^{(0)}(t, y)\,dy \gtrsim \int_{|y|<r_0^{-1}}u_0\,dy$ provided that 
\[
t
\ll
\frac{\int_{y<r_0^{-1}}u_0\,dy}{\int u_0\,dy}.
\]
 The bound for $\|\la x\ra^{\frac{1}{2}}\nabla^{\alpha}u^{(0)}\|_{L^2}$, $|\alpha|\leq 2$, follows also from the explicit kernel representation for the heat kernel $e^{t\triangle}$. 
\\

The difficult part is establishing these bounds for the higher iterates $u^{(j)}$ with $j\geq 1$. We shall proceed by induction, assuming the properties stated in Lemma \ref{lem:bounds} hold for $j-1$ and deducing them for $j$.  This induction will particularly clarify the nature of $T>0$. 

We shall rely in part on the functional analytic framework developed in Theorem 3.1 and Theorem 3.2 of Part 2 of Friedman \cite{MR0454266}: 
let $A(t)$ be an operator valued function, for $t\in [0, T]$, with $A(t)$ acting on some Banach space $X$ (note that $A(t)$ need not be bounded). We suppose that the domains of $A(t)$ are given by $D_A$ (independent of $t\in [0, T]$).  Further consider the following {\bf{Key properties:}}
\begin{itemize}
\item $D_A$ is dense in $X$, and each $A(t)$ is a closed operator.
\item For each $t\in [0, T]$, the resolvent $R(\lambda; A(t))$ of $A(t)$:
$$
R(\lambda; A(t)) \eqdef \left( A(t) - \lambda I \right)^{-1},
$$ 
exists for all $\lambda$ with $\Re(\lambda)\leq 0$.
\item For each $\Re\lambda\leq 0$ we have the bound in the operator norm
$$
\| R(\lambda; A(t))\|\lesssim \frac{1}{|\lambda| + 1}.
$$
\item For any $t$, $\tau$, $s\in [0, T]$, we have a H{\"o}lder estimate
 for the  $\|\cdot\|_{X}$ operator norm
\begin{equation}
\|[A(t)-A(\tau)]A^{-1}(s)\|\lesssim |t-\tau|^{\gamma}.
\label{holderAest}
\end{equation}
This should hold for some $\gamma\in (0,1) $, with the implicit constant depending on $\gamma$. 
\end{itemize}
The implicit constants above (unless otherwise specified) should all be independent of $\lambda$, $t$, $\tau$, $s$ and $\gamma$.
Then following Friedman \cite[Theorem 3.1 and Theorem 3.2]{MR0454266}, there exists a {\it{unique fundamental solution}} $U(t, \tau)\in B(X)$; that is  a strongly continuous operator valued function such that $\text{Range}(U(t, \tau))\subset D_A$ $\forall t, \tau\in [0, T]$, and furthermore 
\[
\partial_t U(t, \tau)+A(t)U(t, \tau) = 0,\quad \tau<t\leq T,\quad U(\tau, \tau) = I.
\]
In the following we will construct suitable operators $A(t)$, then prove that they have the requisite properties to deduce the existence of the fundamental solution.

It may appear natural to use the operator 
\[
A^{(j-1)}(t)\eqdef -\big((-\triangle)^{-1}u^{(j-1)}\big) \triangle + I,
\]
which however is not self-adjoint.  This causes difficulties in establishing the resolvent bounds. 
Instead, we introduce the slightly modified\footnote{Here we omit the superscript $j$ for simplicity} operator
\begin{equation}
\begin{split}
A(t)(u) \eqdef  g_j (-\triangle)(g_j u) + u,
\\
g_j  \eqdef  \big[(-\triangle)^{-1}u^{(j-1)}\big]^{\frac{1}{2}}.
\end{split}
\label{operatorDEF}
\end{equation}
It is not hard to check that this is a self-adjoint operator with domain 
\[
D_A\eqdef \{\tilde{u}\in L^2(\threed) ~ | ~ A(t)(\tilde{u})\in L^2(\threed)\}.
\]
Then it is easily verified that $D_A$ is independent of $t$, in light of the assumptions on $u^{(j-1)}$. 
Note that $\la A(t)\tilde{u}, \tilde{u}\ra_{L^2(\threed)}\geq \|\tilde{u}\|_{L^2(\threed)}^2$ for $\tilde{u}\in D_A$, whence we have 
\[
\|R(\lambda; A(t))\|\leq \frac{1}{1+|\lambda|},\quad \Re\lambda\leq 0,
\]
i.e. the resolvent bound among the key properties is satisfied.  
In particular, $A(\sigma)$, $\sigma \in [0, T]$ generates an analytic semigroup 
\[
e^{-tA(\sigma)}
\]
with the important bounds 
\[
\|A(\sigma)^m e^{-tA(\sigma)}\|\lesssim \frac{1}{t^m},\quad t>0,\quad m=1,2,\ldots.
\]
In order to use the operators $A(t)$, we need to re-formulate \eqref{eqn:modelj} as follows. Let 
\[
\tilde{u}^{(j)}\eqdef  e^{-t} g_j^{-1} u^{(j)}, 
\quad
j \ge 1.
\]
Then we obtain 
\begin{equation}\label{eqn:tildemodelj}
\partial_t\tilde{u}^{(j)} + A(t)\tilde{u}^{(j)}= 
-\frac{\partial_t g_j}{g_j}\tilde{u}^{(j)}
+
\alpha e^{t} \frac{g_{j-1}^2}{g_j} \big(\tilde{u}^{(j-1)}\big)^2
\end{equation}
We then treat 
$
-\frac{\partial_t g_j}{g_j}\tilde{u}^{(j)}
+
\alpha e^{t} \frac{g_{j-1}^2}{g_j} \big(\tilde{u}^{(j-1)}\big)^2
$ 
as source term, 
and apply a bootstrapping argument to recover the $L^2$-based bounds on $\tilde{u}^{(j)}$. The $L^1, L^{2+}$-bounds in turn will follow directly from \eqref{eqn:modelj}. 
\\

\subsection*{Organization of the rest of Section \ref{sec:local}}   In Section \ref{sec:ce} we prove the continuity estimate in \eqref{holderAest}.  Then in Section \ref{sec:control} we prove the uniform bounds on $u^{(j)}$.  After that in Section \ref{sec:mon} we will establish the monotonicity of each $u^{(j)}$ by induction.  Subsequently in Section \ref{sec:ELLcontrol} we prove the pointwise control over the expression  $(-\triangle)^{-1}u^{(j)}$ as in \eqref{boundUinv}.  In the next Section \ref{sec:higherDb} we prove uniform bounds for the higher derivatives.  Section \ref{sec:convergence} then proves the convergence of the $u^{(j)}$.  Finally Section \ref{sec:unq} proves the uniqueness of the solution $u(t,x)$.  \\

In order to construct the fundamental solution $U(t, \tau)$ associated with $A(\sigma)$, we still need to verify the fourth of the key properties, i.e. the H\"{o}lder type bound.

\subsection{The continuity estimate}\label{sec:ce}
Notice that condition \eqref{holderAest} is implied by
\begin{equation}
\|[A(t)-A(\tau)]A^{-1}(\tau)\|\lesssim |t-\tau|^{\gamma},
\quad t, \tau \in [0, T],
\quad \gamma\in (0,1).
\label{holderAestIMP}
\end{equation}
This simplification is explained in Section 3 of Friedman \cite{MR0454266}. 

Consider the identity
\begin{equation}\notag
\big[A(t) - A(\tau)\big]A^{-1}(\tau)
 = 
 \big[g_j(t)(-\triangle)\big(g_j(t)\cdot\big) - g_j(\tau)(-\triangle)\big(g_j(\tau)\cdot\big)\big]
 \circ 
\Psi
 \circ \Phi,
\end{equation}
where we set 
\begin{gather*}
\Psi \eqdef g_j^{-1}(\tau)(-\triangle)^{-1}\big(g_j^{-1}(\tau)\cdot\big),
\\
\Phi\eqdef g_j(\tau)(-\triangle)\big(g_j(\tau)\cdot\big)\circ \big(g_j(\tau)(-\triangle)\big(g_j(\tau)\cdot\big) + I\big)^{-1}.
\end{gather*}
Thus $\Phi$ is clearly $L^2$-bounded. Then we decompose 
\begin{eqnarray}
\notag
-\big[A(t) - A(\tau)\big]A^{-1}(\tau)
= g_j(t)\triangle\big([g_j(t)-g_j(\tau)]g_j^{-1}(\tau)\triangle^{-1}\big(g_j^{-1}(\tau)\cdot\big)\big)\circ\Phi
\\
\label{eqn:1a}
+[g_j(t)-g_j(\tau)]\triangle\big(g_j(\tau)g_j^{-1}(\tau)\triangle^{-1}\big(g_j^{-1}(\tau)\cdot\big)\big)\circ\Phi.
\end{eqnarray}
We estimate the two terms on the right separately.  The second term simplifies to
\begin{equation}
\label{eqn:1b}
[g_j(t)-g_j(\tau)]\big(g_j^{-1}(\tau)\cdot\big)\circ\Phi.
\end{equation}
We decompose the first term in \eqref{eqn:1a} further into 
\begin{eqnarray}\notag
g_j(t)\triangle\big([g_j(t)-g_j(\tau)]g_j^{-1}(\tau)\triangle^{-1}\big(g_j^{-1}(\tau)\cdot\big)\big)\circ\Phi
\\
 = \big(g_j(t)[g_j(t)-g_j(\tau)]g_j^{-2}(\tau)\cdot\big)\circ\Phi
\label{eqn:3a}
 \\
+2g_j(t)\nabla\big([g_j(t)-g_j(\tau)]g_j^{-1}(\tau)\big)\cdot\nabla\triangle^{-1}\big(g_j^{-1}(\tau)\cdot\big)\circ\Phi
\label{eqn:3b}
\\
+g_j(t)\triangle\big([g_j(t)-g_j(\tau)]g_j^{-1}(\tau)\big)\cdot\triangle^{-1}\big(g_j^{-1}(\tau)\cdot\big)\circ\Phi.
\label{eqn:3c}
\end{eqnarray}
Thus to prove \eqref{holderAestIMP} it suffices to estimate \eqref{eqn:1b} - \eqref{eqn:3c}.

We will now show that we can estimate \eqref{eqn:1b} and \eqref{eqn:3a} in a similar way.  
In particular, because of the $L^2(\threed)$ boundedness of $\Phi$, it suffices to establish
\begin{equation}\label{eq:red1}
\|[g_j(t)-g_j(\tau)]g_j^{-1}(\tau)\|_{L^{\infty}(\threed)}
\max\{1,
\|g_j(t)g_j^{-1}(\tau)\|_{L^{\infty}}\}
\lesssim |t-\tau|^{\gamma}, \quad
\gamma \in (0,1).
\end{equation}
Note that $u^{(j-1)}$ will satisfy \eqref{asymptoticU} by the induction assumption which yields from \eqref{operatorDEF} that
$
\|g_j(t)g_j^{-1}(\tau)\|_{L^{\infty}(\threed)}  \lesssim 1.
$
We thus reduce to showing that
\begin{equation}\label{eqn:4}
\big|\la x \ra\big[(-\triangle)^{-1}u^{(j-1)}(t, x) - (-\triangle)^{-1}u^{(j-1)}(\tau, x)\big]\big|\lesssim |t-\tau|^\gamma.
\end{equation}
Both \eqref{eqn:1b} and \eqref{eqn:3a} will then satisfy the bound \eqref{holderAestIMP}.  Note that without loss of generality below we can assume that $|t-\tau| \le 1$ (since we are proving local existence).

\begin{lemma}\label{lem:red1}
Under the hypotheses of Lemma~\ref{lem:bounds} we have the bound 
\[
\big|\la x \ra\big[(-\triangle)^{-1}u^{(j-1)}(t, x) - (-\triangle)^{-1}u^{(j-1)}(\tau, x)\big]\big|\lesssim |t-\tau|^\gamma
\]
for some absolute constant $\gamma>0$. The implied constant depends on the quantities $D_i$ in Lemma~\ref{lem:bounds}.
\end{lemma}

\begin{proof}(Lemma~\ref{lem:red1})
Pick some $\beta>0$. Recalling \eqref{NewtonFormula3D} with $u = u^{(j-1)}$, we can write
\begin{equation}\nonumber\begin{split}
(-\triangle)^{-1}u^{(j-1)}(t, x) 
=&
\frac{1}{4 \pi|x|}\int \chi_{|t-\tau|^{\beta}\leq|y|\leq |x|}u^{(j-1)}(t, y)\,dy 
\\
&+ \int \chi_{|y|\geq \max\{|t-\tau|^{\beta}, |x|\}}\frac{u^{(j-1)}(t, y)}{4 \pi |y|}dy
\\
&+O(\la x \ra^{-1}|t-\tau|^{\frac{\beta}{2}}\|u^{(j-1)}(t)\|_{L^{2}(\threed)}).
\end{split}\end{equation}
where $\chi_{a\leq \cdot\leq b} = \phi(\frac{|y|}{a})-\phi(\frac{|y|}{b})$, and $\chi_{|y|\geq a} = \phi(\frac{|y|}{a})$ for $\phi(x)$ a smooth cutoff which equals $1$ for $|x|\geq 2$ and vanishes identically for $|x|\leq 1$.   
Note that the $O(\cdot)$ terms result from applying the Cauchy-Schwarz inequality to the terms which were not written.  

Then we obtain 
$$
(-\triangle)^{-1}u^{(j-1)}(t, x) - (-\triangle)^{-1}u^{(j-1)}(\tau, x)
= I_1 +  I_2 +  I_3.
$$
Here 
\begin{equation}\nonumber\begin{split}
I_1
=&\int\chi_{|y|\geq \max\{|t-\tau|^{\beta},|x|\}}\frac{[u^{(j-1)}(t, y) - u^{(j-1)}(\tau, y)]}{4 \pi|y|}dy,
\\
I_2 =&\frac{1}{4 \pi |x|}\int\chi_{|t-\tau|^{\beta}\leq|y|\leq |x|}[u^{(j-1)}(t, y) - u^{(j-1)}(\tau, y)]\,dy,
\\
I_3 = & O\left(\la x \ra^{-1}|t-\tau|^{\frac{\beta}{2}}\left\{\|u^{(j-1)}(t)\|_{L^{2}(\threed)}+ \|u^{(j-1)}(\tau)\|_{L^{2}(\threed)} \right\}\right).
\end{split}\end{equation}
Now we refer to the equation defining $u^{(j-1)}(t,\cdot)$ in \eqref{eqn:modelj}, whence we get 
\[
u^{(j-1)}(t, \cdot) - u^{(j-1)}(\tau, \cdot)=\int_t^\tau\big[(-\triangle)^{-1}u^{(j-2)}(s, \cdot)\triangle u^{(j-1)}(s,\cdot) + \alpha\big(u^{(j-2)}(s, \cdot)\big)^2\big] ds.
\]
Thus we obtain the identity
\begin{equation}\nonumber\begin{split}
I_1
=&\int_t^\tau\int\chi_{|y|\geq \max\{|t-\tau|^{\beta},|x|\}}\frac{(-\triangle)^{-1}u^{(j-2)}(s, y)\triangle u^{(j-1)}(s,y)}{4\pi  |y|}\, dy ds
\\
& +
\int_t^\tau\int\chi_{|y|\geq \max\{|t-\tau|^{\beta},|x|\}}\frac{ \alpha\big(u^{(j-2)}(s, y)\big)^2}{4\pi  |y|}\, dy ds,
\end{split}\end{equation}
and similarly for $I_2$.

We will use the following general monotonicity estimate.  Suppose that $u(x)\ge 0$ is any radial monotonically decreasing function.  Further assume that
$$
\int_{0}^{r_0} u (r) r^2 dr \leq D_3.
$$
Then by monotonicity
\begin{equation}
\label{monEST}
u(r_0)\leq \frac{ 3 D_3}{ r_0^3}.
\end{equation}
We will use this estimate several times below.

With \eqref{monEST} applied to $u^{(j-2)}$, we have that
$
|u^{(j-2)}(t, x)|\lesssim |x|^{-3}.
$
Using this as well as performing integrations by parts, we obtain
$$
\left|
I_1
\right|
\lesssim |t-\tau|^{1-4\beta}\la x \ra^{-1},
$$
and similarly for  $I_2$. We conclude that 
\[
(-\triangle)^{-1}u^{(j-1)}(t, x) - (-\triangle)^{-1}u^{(j-1)}(\tau, x) = O\big(|t-\tau|^{1-4\beta}\la x \ra^{-1} + \la x \ra^{-1}|t-\tau|^{\frac{\beta}{2}}\big),
\]
where the implicit constant depends on $\|u^{(j-1)}\|_{X}+\|u^{(j-2)}\|_{X}$.  
Picking $\beta<\frac{1}{4}$, we obtain \eqref{eqn:4} with $\gamma = \min\{1-4\beta, \frac{\beta}{2}\}$.
\end{proof}

This proves the \eqref{eqn:4} and hence it proves the $L^2$ operator bound for 
\eqref{eqn:1b} and \eqref{eqn:3a}.  

We prove next the $L^2$ bound for \eqref{eqn:3b}.  
Decompose this term further as
\begin{align}
&\nonumber2g_j(t)\nabla\big([g_j(t)-g_j(\tau)]g_j^{-1}(\tau)\big)\cdot\nabla\triangle^{-1}\big(g_j^{-1}(\tau)\cdot\big)\circ\Phi
\\&= g_j(t) \big[g_j^{-1}(t)\nabla(-\triangle)^{-1}u^{(j-1)}(t, \cdot) - g_j^{-1}(\tau)\nabla(-\triangle)^{-1}u^{(j-1)}(\tau, \cdot)\big]\nonumber\\
&\hspace{7.5cm}\cdot g_j^{-1}(\tau)\nabla\triangle^{-1}\big(g_j^{-1}(\tau)\cdot\big)\circ\Phi\label{eq:red3}\\
&-g_j(t)\big([g_j(t)-g_j(\tau)]\nabla\triangle^{-1}u^{(j-1)}(\tau)g_j^{-3}(\tau)\big)\cdot\nabla\triangle^{-1}\big(g_j^{-1}(\tau)\cdot\big)\circ\Phi\label{eq:red4}
\end{align}
{\it{The estimate for \eqref{eq:red3}}}. Split it into 
\begin{equation}\label{eqn: 5}\begin{split}
&g_j(t) \big[g_j^{-1}(t)\nabla(-\triangle)^{-1}u^{(j-1)}(t, \cdot) - g_j^{-1}(\tau)\nabla(-\triangle)^{-1}u^{(j-1)}(\tau, \cdot)\big]\\
&\hspace{7.5cm}\cdot g_j^{-1}(\tau)\nabla\triangle^{-1}\big(g_j^{-1}(\tau)\cdot\big)\circ\Phi\\
& = \big[\nabla(-\triangle)^{-1}u^{(j-1)}(t, \cdot) - \nabla(-\triangle)^{-1}u^{(j-1)}(\tau, \cdot)\big]\\
&\hspace{7.5cm}\cdot g_j^{-1}(\tau)\nabla\triangle^{-1}\big(g_j^{-1}(\tau)\cdot\big)\circ\Phi\\
&+g_j^{-1}(\tau) \big[(g_j(\tau) - g_j(t)\nabla(-\triangle)^{-1}u^{(j-1)}(\tau, \cdot)\big]\\
&\hspace{7.5cm}\cdot g_j^{-1}(\tau)\nabla\triangle^{-1}\big(g_j^{-1}(\tau)\cdot\big)\circ\Phi\\
\end{split}\end{equation}
Commence with the first term on the right: as before, the trick consists in splitting into a region of small $|x|$ and one of large $|x|$: for $\beta>0$ to be determined, write 
\begin{equation}\nonumber\begin{split}
&\big[\nabla(-\triangle)^{-1}u^{(j-1)}(t, \cdot) - \nabla(-\triangle)^{-1}u^{(j-1)}(\tau, \cdot)\big]\cdot g_j^{-1}(\tau)\nabla\triangle^{-1}\big(g_j^{-1}(\tau)\cdot\big)\circ\Phi\\
& =\frac{x}{4\pi|x|^{3}}\int_{|y|\leq|x|} \chi_{|y|\lesssim |t-\tau|^{\beta}}[u^{(j-1)}(t, y) -u^{(j-1)}(\tau, y)]\,dy\big]\cdot g_j^{-1}(\tau)\nabla\triangle^{-1}\big(g_j^{-1}(\tau)\cdot\big)\circ\Phi\\
&+\frac{x}{4\pi|x|^{3}}\int_{|y|\leq|x|}\chi_{|y|\gtrsim |t-\tau|^{\beta}}[u^{(j-1)}(t, y) -u^{(j-1)}(\tau, y)]\,dy\big]\cdot g_j^{-1}(\tau)\nabla\triangle^{-1}\big(g_j^{-1}(\tau)\cdot\big)\circ\Phi\\
 \end{split}\end{equation}
Then note that since 
\begin{equation}\label{eq:red4_5}
\big|\big(g_j^{-1}(\tau)\nabla\triangle^{-1}\big(g_j^{-1}(\tau)\cdot\big)\circ\Phi\big)(\tilde{u})\big|(x)\lesssim \max\{|x|^{-\frac{1}{2}}, |x|^{\frac{1}{2}}\}\|\tilde{u}\|_{L^2},
\end{equation}
we get 
\begin{equation}\nonumber\begin{split}
&\big|\frac{x}{4\pi|x|^{3}}\int_{|y|\leq|x|}\big(\chi_{|y|\lesssim |t-\tau|^{\beta}}[u^{(j-1)}(t, y) -u^{(j-1)}(\tau, y)]\,dy\big]\cdot g_j^{-1}(\tau)\nabla\triangle^{-1}\big(g_j^{-1}(\tau)\cdot\big)\circ\Phi\big)(\tilde{u})\big|\\&\lesssim \min\{|x|^{-\frac{5}{4}}, |x|^{-\frac{3}{2}}\}|t-\tau|^{\frac{\beta}{4}}\|\tilde{u}\|_{L^2}
\end{split}\end{equation}
This unfortunately fails logarithmically to be in $L^2(\threed)$. To remedy this, note that for radial $\tilde{u}\in L^2(\threed)$, we get due to Newton's formula 
\[
\chi_{|x|\sim 2^k}\big(g_j^{-1}(\tau)\nabla\triangle^{-1}\big(g_j^{-1}(\tau)\tilde{u}\big)\big)(x)\lesssim 2^{\frac{k}{2}}\big[\sum_{\tilde{j}\leq k}2^{2(\tilde{j}-k)}\|\chi_{|x|\sim 2^{\tilde{j}}}\tilde{u}\|_{L^2}+2^{-2k}\|\chi_{|x|\lesssim 1}\tilde{u}\|_{L^2}\big]
\]
where $\tilde{j}, k\in \mathbb{N}$. Indeed, we have 
\begin{align*}
|\chi_{|x|\sim 2^k}\big(g_j^{-1}(\tau)\nabla\triangle^{-1}\big(g_j^{-1}(\tau)\tilde{u}\big)\big)(x)|&\lesssim 
2^{\frac{k}{2}}\chi_{|x|\sim 2^k}|x|^{-2}\int_{|y|\leq|x|}g_j^{-1}(\tau)\tilde{u}\,dy\\
&\lesssim 2^{\frac{k}{2}}\sum_{\tilde{j}\leq k}2^{-2k+2\tilde{j}}\|\chi_{|x|\sim 2^{\tilde{j}}}\tilde{u}\|_{L^2} +2^{-\frac{3k}{2}}\|\chi_{|x|\lesssim 1}\tilde{u}\|_{L^2}
\end{align*}
Here we have used the Cauchy-Schwarz inequality to estimate the inherent integral term.
Then note that 
\begin{equation}\label{eq:red5}\begin{split}
&\|\sum_{k\geq 0} |x|^{-2}\chi_{|x|\sim 2^k}\big(g_j^{-1}(\tau)\nabla\triangle^{-1}\big(g_j^{-1}(\tau)\tilde{u}\big)\big)\|_{L^2}^2\\
&\lesssim \sum_{k\geq 0}\big[\sum_{j\leq k}2^{2(j-k)}\|\chi_{|x|\sim 2^j}\tilde{u}\|_{L^2}+2^{-2k}\|\chi_{|x|\lesssim 1}\tilde{u}\|_{L^2}\big]^2\\
&\lesssim \|\tilde{u}\|_{L^2}^2
\end{split}\end{equation}
We thus get (here $\|\cdot\|$ refers to the $L^2$-operator bound)
\begin{equation}\nonumber\begin{split}
&\|\frac{x}{4\pi|x|^{3}}\int_{|y|\leq|x|} \chi_{|y|\lesssim |t-\tau|^{\beta}}[u^{(j-1)}(t, y) -u^{(j-1)}(\tau, y)]\,dy\big]\cdot g_j^{-1}(\tau)\nabla\triangle^{-1}\big(g_j^{-1}(\tau)\cdot\big)\circ\Phi\|\\
&\lesssim |t-\tau|^{\frac{\beta}{4}}
\end{split}\end{equation}
On the other hand, for the second term above, we again use the equation satisfied by $u^{(j-1)}$, which furnishes 
\begin{equation}\nonumber\begin{split}
&\frac{x}{4\pi|x|^{3}}\int_{|y|\leq|x|}\chi_{|y|\gtrsim |t-\tau|^{\beta}}[u^{(j-1)}(t, y) -u^{(j-1)}(\tau, y)]\,dy\big]\\
& = \frac{x}{4\pi|x|^{3}}\int_t^\tau\int_{|y|\leq|x|}\chi_{|y|\gtrsim |t-\tau|^{\beta}}[(-\triangle)^{-1}u^{(j-2)}(\tau, \cdot)\triangle u^{(j-1)}(\tau,\cdot) + \alpha\big(u^{(j-2)}(\tau, \cdot)\big)^2]\,dyds\big],
\end{split}\end{equation}
whence we have (with $\|.\|$ denoting $L^2$-operator norm)
\begin{equation}\nonumber\begin{split}
&\|\frac{x}{4\pi|x|^{3}}\int_{|y|\leq|x|}\chi_{|y|\gtrsim |t-\tau|^{\beta}}[u^{(j-1)}(t, y) -u^{(j-1)}(\tau, y)]\,dy\big]\cdot g_j^{-1}(\tau)\nabla\triangle^{-1}\big(g_j^{-1}(\tau)\cdot\big)\circ\Phi\|\\
&\lesssim |t-\tau|^{1-2\beta}
\end{split}\end{equation}
In summary, we obtain 
\[
\|\big[\nabla(-\triangle)^{-1}u^{(j-1)}(t, \cdot) - \nabla(-\triangle)^{-1}u^{(j-1)}(\tau, \cdot)\big]\cdot g_j^{-1}(\tau)\nabla\triangle^{-1}\big(g_j^{-1}(\tau)\cdot\big)\circ\Phi\|\lesssim |t-\tau|^{\gamma}
\]
with $\gamma = \min\{\frac{\beta}{4}, 1-2\beta\}$ where we take $\beta<\frac{1}{2}$.

The second term in \eqref{eqn: 5} on the other hand can be estimated by using 
\[
\big|g_j(t)g_j^{-2}(\tau)[(g_j(\tau) - g_j(t)]\big| \lesssim |t-\tau|^{\gamma}
\]
with $\gamma>0$, see \eqref{eq:red1}, and the bound
\begin{equation}\label{eq:red7}
\|\nabla(-\triangle)^{-1}u^{(j-1)}(\tau, \cdot)\cdot g_j^{-1}(\tau)\nabla\triangle^{-1}\big(g_j^{-1}(\tau)\cdot\big)\circ\Phi\|\lesssim 1
\end{equation}
which follows from \eqref{eq:red5}: indeed, we have 
\[
|\nabla(-\triangle)^{-1}u^{(j-1)}(\tau, x)|\lesssim \min\{|x|^{-\frac{1}{2}}, |x|^{-2}\}
\]
and so one finds, using also  \eqref{eq:red4_5}
\begin{align*}
&\|\big(\nabla(-\triangle)^{-1}u^{(j-1)}(\tau, \cdot)\cdot g_j^{-1}(\tau)\nabla\triangle^{-1}\big(g_j^{-1}(\tau)\cdot\big)\circ\Phi\big)(u)\|_{L^2}\\
&\lesssim \|\sum_{k\geq 0}|x|^{-2}\chi_{|x|\sim 2^k}\big(g_j^{-1}(\tau)\nabla\triangle^{-1}\big(g_j^{-1}(\tau)\cdot\big)\circ\Phi\big)(u)\|_{L^2}\\
&+\|\chi_{|x|\lesssim 1}|x|^{-\frac{1}{2}}\big(g_j^{-1}(\tau)\nabla\triangle^{-1}\big(g_j^{-1}(\tau)\cdot\big)\circ\Phi\big)(u)\|_{L^2}\\
&\lesssim \|u\|_{L^2}
\end{align*}

{\it{The estimate for \eqref{eq:red4}.}}  Due to \eqref{eq:red1} as well as \eqref{eq:red7}, we have 
\begin{align*}
&\|g_j(t)\big([g_j(t)-g_j(\tau)]\nabla\triangle^{-1}u^{(j-1)}(\tau)g_j^{-3}(\tau)\big)\cdot\nabla\triangle^{-1}\big(g_j^{-1}(\tau)\cdot\big)\circ\Phi\|\\
&\lesssim \|g_j(t)[g_j(t)-g_j(\tau)]g_j^{-2}(\tau)\|_{L^\infty}\|\nabla(-\triangle)^{-1}u^{(j-1)}(\tau, \cdot)\cdot g_j^{-1}(\tau)\nabla\triangle^{-1}\big(g_j^{-1}(\tau)\cdot\big)\circ\Phi\|\\
&\lesssim |t-\tau|^{\gamma}
\end{align*}

This completes the $L^2$  bound for \eqref{eqn:3b}.  

Lastly we prove the $L^2$ operator bound for \eqref{eqn:3c}.  
We decompose it into 
\begin{equation}\label{eqn: 6}\begin{split}
&g_j(t)\triangle\big([g_j(t)-g_j(\tau)]g_j^{-1}(\tau)\big)\cdot\triangle^{-1}\big(g_j^{-1}(\tau)\cdot\big)\circ\Phi\\
&= g_j(t)\big([g_j(t)-g_j(\tau)]u^{(j-1)}g_j^{-3}(\tau)\big)\cdot\triangle^{-1}\big(g_j^{-1}(\tau)\cdot\big)\circ\Phi\\
&+g_j(t)\big([g_j(t)-g_j(\tau)]\big(\nabla(-\triangle)^{-1}u^{(j-1)}\big)^2g_j^{-5}(\tau)\big)\cdot\triangle^{-1}\big(g_j^{-1}(\tau)\cdot\big)\circ\Phi\\
&+g_j(t)\big(\nabla[g_j(t)-g_j(\tau)]\cdot\nabla(-\triangle)^{-1}u^{(j-1)}(\tau)g_j^{-3}(\tau)\big)\cdot\triangle^{-1}\big(g_j^{-1}(\tau)\cdot\big)\circ\Phi\\
&+g_j(t)\big(\triangle[g_j(t)-g_j(\tau)]g_j^{-1}(\tau)\big)\cdot\triangle^{-1}\big(g_j^{-1}(\tau)\cdot\big)\circ\Phi\\
\end{split}\end{equation}
{\it{First term on right hand side of \eqref{eqn: 6}}}. Here we need to exploit the precise structure of 
\[
 \Phi = g_j(\tau)(-\triangle)\big(g_j(\tau)\cdot\big)\circ \big(g_j(\tau)(-\triangle)\big(g_j(\tau)\cdot\big) + I\big)^{-1}
 \]
 Hence we get for $\tilde{u}\in L^2(\threed)$
 \begin{equation}\nonumber\begin{split}
 &\big(g_j(t)\big([g_j(t)-g_j(\tau)]u^{(j-1)}g_j^{-3}(\tau)\big)\cdot\triangle^{-1}\big(g_j^{-1}(\tau)\cdot\big)\circ\Phi\big)(\tilde{u})\\
 &= g_j(t)\big([g_j(t)-g_j(\tau)]u^{(j-1)}g_j^{-2}(\tau)\big)\cdot\big(g_j(\tau)(-\triangle)\big(g_j(\tau)\cdot\big) + I\big)^{-1}(\tilde{u})\\
 &=\chi_{|x|\lesssim1}g_j(t)\big([g_j(t)-g_j(\tau)]u^{(j-1)}g_j^{-2}(\tau)\big)\cdot\big(g_j(\tau)(-\triangle)\big(g_j(\tau)\cdot\big) + I\big)^{-1}(\tilde{u})\\
 &+\chi_{|x|\gtrsim 1}g_j(t)\big([g_j(t)-g_j(\tau)]u^{(j-1)}g_j^{-2}(\tau)\big)\cdot\big(g_j(\tau)(-\triangle)\big(g_j(\tau)\cdot\big) + I\big)^{-1}(\tilde{u})\\
 \end{split}\end{equation}
The second term here can be immediately estimated since 
\[
\|\chi_{|x|\gtrsim 1}g_j(t)\big([g_j(t)-g_j(\tau)]u^{(j-1)}g_j^{-2}(\tau)\|_{L_x^\infty}\lesssim |t-\tau|^{\gamma}
\]
for $\gamma>0$ using Lemma~\ref{lem:red1} as well as \eqref{monEST}, and hence 
\begin{equation}\nonumber\begin{split}
&\|\chi_{|x|\gtrsim 1}g_j(t)\big([g_j(t)-g_j(\tau)]u^{(j-1)}g_j^{-2}(\tau)\big)\cdot\big(g_j(\tau)(-\triangle)\big(g_j(\tau)\cdot\big) + I\big)^{-1}(\tilde{u})\|_{L^2}\\
&\lesssim  |t-\tau|^{\gamma}\|\big(g_j(\tau)(-\triangle)\big(g_j(\tau)\cdot\big) + I\big)^{-1}(\tilde{u})\|_{L^2}\lesssim  |t-\tau|^{\gamma}\|\tilde{u}\|_{L^2}
\end{split}\end{equation}
For the first term above, denoting $v = \chi_{|x|\lesssim1}\big(g_j(\tau)(-\triangle)\big(g_j(\tau)\cdot\big) + I\big)^{-1}(\tilde{u})$, observe that by elliptic regularity theory we have 
\[
\|v\|_{W^{2,2}}\lesssim \|\tilde{u}\|_{L^2},
\]
whence we have 
\begin{equation}\nonumber\begin{split}
&\|\chi_{|x|\lesssim1}g_j(t)\big([g_j(t)-g_j(\tau)]u^{(j-1)}g_j^{-2}(\tau)\big)\cdot\big(g_j(\tau)(-\triangle)\big(g_j(\tau)\cdot\big) + I\big)^{-1}(\tilde{u})\|_{L^2}\\
&\lesssim \|g_j(t)\big([g_j(t)-g_j(\tau)]g_j^{-2}(\tau)\|_{L^\infty}\|u^{(j-1)}\|_{L^2}\|\tilde{u}\|_{L^2}\lesssim|t-\tau|^{\gamma}\|\tilde{u}\|_{L^2}
\end{split}\end{equation}

{\it{Second term on right hand side of \eqref{eqn: 6}}}. Write this, applied to $\tilde{u}\in L^2(\threed)$, as 
\begin{equation}\nonumber\begin{split}
&g_j(t)\big([g_j(t)-g_j(\tau)]\big(\nabla(-\triangle)^{-1}u^{(j-1)}\big)^2g_j^{-4}(\tau)\cdot \big(g_j(\tau)(-\triangle)\big(g_j(\tau)\cdot\big) + I\big)^{-1}(\tilde{u})\\
&= \chi_{|x|\lesssim1}g_j(t)\big([g_j(t)-g_j(\tau)]\big(\nabla(-\triangle)^{-1}u^{(j-1)}\big)^2g_j^{-4}(\tau)\cdot \big(g_j(\tau)(-\triangle)\big(g_j(\tau)\cdot\big) + I\big)^{-1}(\tilde{u})\\
&+\chi_{|x|\gtrsim1}g_j(t)\big([g_j(t)-g_j(\tau)]\big(\nabla(-\triangle)^{-1}u^{(j-1)}\big)^2g_j^{-4}(\tau)\cdot \big(g_j(\tau)(-\triangle)\big(g_j(\tau)\cdot\big) + I\big)^{-1}(\tilde{u})\\
\end{split}\end{equation}
For the first term on the right, we have 
\[
 \chi_{|x|\lesssim1}\big(\nabla(-\triangle)^{-1}u^{(j-1)}\big)^2\sim  \chi_{|x|\lesssim1}|x|^{-1}
\]
in light of Newton's formula \eqref{NewtonFormula3D}, H\"{o}lder's inequality and the fact that $u^{(j-1)}\in L^2(\threed)$. Then reasoning as for the first term of \eqref{eqn: 6}, we obtain
\begin{equation}\nonumber\begin{split}
&\| \chi_{|x|\lesssim1}g_j(t)\big([g_j(t)-g_j(\tau)]\big(\nabla(-\triangle)^{-1}u^{(j-1)}\big)^2g_j^{-4}(\tau)\cdot \big(g_j(\tau)(-\triangle)\big(g_j(\tau)\cdot\big) + I\big)^{-1}(\tilde{u})\|_{L^2}\\
&\lesssim |t-\tau|^{\gamma}\|\big(\nabla(-\triangle)^{-1}u^{(j-1)}\big)^2\|_{L^2}\|\big(g_j(\tau)(-\triangle)\big(g_j(\tau)\cdot\big) + I\big)^{-1}(\tilde{u})\|_{L^\infty}\\
&\lesssim |t-\tau|^{\gamma}\|\tilde{u}\|_{L^2}
\end{split}\end{equation}
For the term 
\[
\chi_{|x|\gtrsim1}g_j(t)\big([g_j(t)-g_j(\tau)]\big(\nabla(-\triangle)^{-1}u^{(j-1)}\big)^2g_j^{-4}(\tau)\cdot \big(g_j(\tau)(-\triangle)\big(g_j(\tau)\cdot\big) + I\big)^{-1}(\tilde{u}), 
\]
we can estimate it by 
\begin{align*}
&\|\chi_{|x|\gtrsim1}g_j(t)\big([g_j(t)-g_j(\tau)]\big(\nabla(-\triangle)^{-1}u^{(j-1)}\big)^2g_j^{-4}(\tau)\cdot \big(g_j(\tau)(-\triangle)\big(g_j(\tau)\cdot\big) + I\big)^{-1}(\tilde{u})\|_{L^2}\\
&\lesssim \|g_j(t)[g_j(t)-g_j(\tau)]g_j^{-2}(\tau)\|_{L^\infty}\|\chi_{|x|\gtrsim1}\nabla(-\triangle)^{-1}u^{(j-1)}\big)^2g_j^{-2}(\tau)\|_{L^\infty}\\
&\hspace{6.5cm}\cdot\|\big(g_j(\tau)(-\triangle)\big(g_j(\tau)\cdot\big) + I\big)^{-1}(\tilde{u})\|_{L^2}\\
&\lesssim |t-\tau|^{\gamma}\|\tilde{u}\|_{L^2},
\end{align*}
having used the bounds \eqref{eq:red1} as well as $|\nabla(-\triangle)^{-1}u^{(j-1)}(\cdot, x)|\lesssim |x|^{-2}$.\\
{\it{Third term of right hand side of \eqref{eqn: 6}}}. Write it (applied to $\tilde{u}\in L^2(\threed)$) as
\begin{equation}\nonumber\begin{split}
&g_j(t)\big[ g_j^{-1}(t)\nabla(-\triangle)^{-1}u_j(t) - g_j^{-1}(\tau)\nabla(-\triangle)^{-1}u_j(\tau)\big]\\&\hspace{1cm}\cdot\nabla(-\triangle)^{-1}u^{(j-1)}(\tau) g_j^{-2}(\tau)\big(g_j(\tau)(-\triangle)\big(g_j(\tau)\cdot\big) + I\big)^{-1}(\tilde{u})\\
& = \big[\nabla(-\triangle)^{-1}u_j(t) - \nabla(-\triangle)^{-1}u_j(\tau)\big]\\&\hspace{2cm}\cdot\nabla(-\triangle)^{-1}u^{(j-1)}(\tau)g_j^{-2}(\tau)\big(g_j(\tau)(-\triangle)\big(g_j(\tau)\cdot\big) + I\big)^{-1}(\tilde{u})\\
&+g_j(t)[g_j^{-1}(t) -  g_j^{-1}(\tau)]\nabla(-\triangle)^{-1}u_j(\tau)\\&\hspace{2cm}\cdot\nabla(-\triangle)^{-1}u^{(j-1)}(\tau) g_j^{-2}(\tau)\big(g_j(\tau)(-\triangle)\big(g_j(\tau)\cdot\big) + I\big)^{-1}(\tilde{u})\\
\end{split}\end{equation}
For both of these one splits into the regions $|x|\lesssim 1$, $|x|\gtrsim 1$, and one further decomposes the integrals giving $\nabla(-\triangle)^{-1}u_j(t) - \nabla(-\triangle)^{-1}u_j(\tau)$ into the regions $|y|\lesssim |t-\tau|^{\beta}$, $|y|\gtrsim |t-\tau|^{\beta}$. In the region $|x|\lesssim 1$, one places 
\[
\big(g_j(\tau)(-\triangle)\big(g_j(\tau)\cdot\big) + I\big)^{-1}(\tilde{u})
\]
into $L^\infty$ and the remaining product directly into $L^2$, as in our estimates for \eqref{eqn:3b}. In the region $|x|\gtrsim 1$, one easily checks directly that 
\begin{equation}\nonumber\begin{split}
&\|\chi_{|x|\gtrsim 1}\big[\nabla(-\triangle)^{-1}u_j(t) - \nabla(-\triangle)^{-1}u_j(\tau)\big]\cdot\nabla(-\triangle)^{-1}u^{(j-1)}(\tau)g_j^{-2}(\tau)\|_{L^{\infty}},\\
&\lesssim |t-\tau|^{\gamma}
\end{split}\end{equation}
\begin{equation}\nonumber\begin{split}
\|\chi_{|x|\gtrsim 1}g_j(t)[g_j^{-1}(t) -  g_j^{-1}(\tau)]\nabla(-\triangle)^{-1}u_j(\tau)\cdot\nabla(-\triangle)^{-1}u^{(j-1)}(\tau) g_j^{-2}(\tau)\|_{L^{\infty}}\lesssim |t-\tau|^{\gamma}
\end{split}\end{equation}
from which the desired estimate easily follows. 

{\it{The fourth term of \eqref{eqn: 6}}}. We expand it into 
\begin{align}
&\nonumber g_j(t)\big(\triangle[g_j(t)-g_j(\tau)]g_j^{-1}(\tau)\big)\cdot\triangle^{-1}\big(g_j^{-1}(\tau)\cdot\big)\circ\Phi\\
&= [u^{(j-1)}(t, \cdot) - u^{(j-1)}(\tau, \cdot)] ]g_j^{-1}(\tau)\big)\cdot\triangle^{-1}\big(g_j^{-1}(\tau)\cdot\big)\circ\Phi\label{eq:red8}\\
&+g_j(t)\big[(\nabla(-\triangle)^{-1}u^{(j-1)}(t, \cdot))^2g_j^{-3}(t) - (\nabla(-\triangle)^{-1}u^{(j-1)}(\tau, \cdot))^2g_j^{-3}(\tau)\big]\label{eq:red9}
\\ &\hspace{7cm} \cdot g_j^{-1}(\tau)\cdot\triangle^{-1}\big(g_j^{-1}(\tau)\cdot\big)\circ\Phi\nonumber
\end{align}
\it{The expression \eqref{eq:red8}}. It appears to require a somewhat different method, as we no longer average the difference term $[u^{(j-1)}(t, \cdot) - u^{(j-1)}(\tau, \cdot)]$ over $x$. We proceed in a number of steps, taking advantage of frequency localization: observe that due to the non-increasing nature of $u^{(j-1)}$, as well as radiality, we have (for some $\beta_1>0$ to be chosen)
\[
\|\nabla\big[\chi_{|t-\tau|^{-\beta_1}\gtrsim |x|\gtrsim |t-\tau|^{\beta_1}}u^{(j-1)}(\tau, \cdot)\big]\|_{L^1}\lesssim |t-\tau|^{-\frac{7\beta_1}{2}},\,s\in [0, T]
\]
Now let $P_{<a}$, $P_{\geq b}$ etc  be standard Littlewood-Paley frequency cutoffs localizing to frequencies $\lesssim 2^a$, $\gtrsim 2^{b}$, respectively. Then we have 
\[
\|P_{\geq -\beta_2\log_2|t-\tau|}\big[\chi_{|t-\tau|^{-\beta_1}\gtrsim |x|\gtrsim |t-\tau|^{\beta_1}}u^{(j-1)}(\tau, \cdot)\big]\|_{L^1}\lesssim |t-\tau|^{\beta_2-\frac{7\beta_1}{2}}
\]
Interpolating this with the bound (using radiality and monotonicity of $u^{(j-1)}$)
\[
\|P_{\geq -\beta_2\log_2|t-\tau|}\big[\chi_{|t-\tau|^{-\beta_1}\gtrsim |x|\gtrsim |t-\tau|^{\beta_1}}u^{(j-1)}(\tau, \cdot)\big]\|_{L^\infty}\lesssim |t-\tau|^{-3\beta_1},
\]
we get 
\begin{equation}\label{eqn: 7}
\|P_{\geq -\beta_2\log_2|t-\tau|}\big[\chi_{|t-\tau|^{-\beta_1}\gtrsim |x|\gtrsim |t-\tau|^{\beta_1}}u^{(j-1)}(\tau, \cdot)\big]\|_{L^2}\lesssim |t-\tau|^{\frac{\beta_2}{2}-\frac{13\beta_1}{4}}
\end{equation}
In order to use this, we decompose (here again $\tilde{u}\in L^2(\threed)$)
\begin{equation}\label{eqn: mess0}\begin{split}
&\big([u^{(j-1)}(t, \cdot) - u^{(j-1)}(\tau, \cdot)] g_j^{-1}(\tau)\big)\cdot\triangle^{-1}\big(g_j^{-1}(\tau)\cdot\big)\circ\Phi\big)(\tilde{u})\\
&=\chi_{|x|\gtrsim |t-\tau|^{-\beta_1}}([u^{(j-1)}(t, \cdot) - u^{(j-1)}(\tau, \cdot)] \big)\big(g_j(\tau)(-\triangle)g_j(\tau)+I\big)^{-1}\tilde{u}\\
&+\chi_{|t-\tau|^{-\beta_1}\gtrsim |x|\gtrsim |t-\tau|^{\beta_1}}([u^{(j-1)}(t, \cdot) - u^{(j-1)}(\tau, \cdot)] \big)\big(g_j(\tau)(-\triangle)g_j(\tau)+I\big)^{-1}\tilde{u}\\
&+\chi_{|x|\lesssim |t-\tau|^{\beta_1}}([u^{(j-1)}(t, \cdot) - u^{(j-1)}(\tau, \cdot)] \big)\big(g_j(\tau)(-\triangle)g_j(\tau)+I\big)^{-1}\tilde{u}\\
\end{split}\end{equation}
For the first term on the right, use 
\begin{equation}\label{eqn: mess1}\begin{split}
&\|\chi_{|x|\gtrsim |t-\tau|^{-\beta_1}}([u^{(j-1)}(t, \cdot) - u^{(j-1)}(\tau, \cdot)] \big)\big(g_j(\tau)(-\triangle)g_j(\tau)+I\big)^{-1}\tilde{u}\|_{L^2}\\
&\lesssim \|\chi_{|x|\gtrsim |t-\tau|^{-\beta_1}}([u^{(j-1)}(t, \cdot) - u^{(j-1)}(\tau, \cdot)] \|_{L^{\infty}}
\|\big(g_j(\tau)(-\triangle)g_j(\tau)+I\big)^{-1}\tilde{u}\|_{L^2}\\
&\lesssim |t-\tau|^{3\beta_1}\|\tilde{u}\|_{L^2}
\end{split}\end{equation}
On the other hand, for the last term above, we use H\"{o}lder's inequality and $L^{2+}$-control: 
We have 
\begin{equation}\nonumber\begin{split}
&\|\chi_{|x|\lesssim |t-\tau|^{\beta_1}}([u^{(j-1)}(t, \cdot) - u^{(j-1)}(\tau, \cdot)] \big)\big(g_j(\tau)(-\triangle)g_j(\tau)+I\big)^{-1}\tilde{u}\|_{L^2}\\
&\lesssim  |t-\tau|^{\nu\beta_1}\|u^{(j-1)}(t, \cdot) - u^{(j-1)}(\tau, \cdot)\|_{L^{2+}}\|\chi_{|x|\lesssim 1}\big(g_j(\tau)(-\triangle)g_j(\tau)+I\big)^{-1}\tilde{u}\|_{L^\infty}\\
&\lesssim  |t-\tau|^{\nu\beta_1}\|u^{(j-1)}(t, \cdot) - u^{(j-1)}(\tau, \cdot)\|_{L^{2+}}\|\tilde{u}\|_{L^2}\lesssim |t-\tau|^{\nu\beta_1}\|\tilde{u}\|_{L^2}.
\end{split}\end{equation}
We conclude that in \eqref{eqn: mess0}, we have reduced to estimating the middle term on the right hand side, for which \eqref{eqn: 7} will come handy. Write 
\begin{equation}\label{eqn: 8}\begin{split}
&\chi_{|t-\tau|^{-\beta_1}\gtrsim |x|\gtrsim |t-\tau|^{\beta_1}}([u^{(j-1)}(t, \cdot) - u^{(j-1)}(\tau, \cdot)] \big)\big(g_j(\tau)(-\triangle)g_j(\tau)+I\big)^{-1}\tilde{u}\\
&=P_{<-\beta_2\log_2|t-\tau|}\big[\chi_{|t-\tau|^{-\beta_1}\gtrsim |x|\gtrsim |t-\tau|^{\beta_1}}([u^{(j-1)}(t, \cdot) - u^{(j-1)}(\tau, \cdot)] \big]
\\&\hspace{5cm}\cdot\tilde{\chi}_{|x|\lesssim |t-\tau|^{-\beta_1}}\big(g_j(\tau)(-\triangle)g_j(\tau)+I\big)^{-1}\tilde{u}\\
&+P_{\geq-\beta_2\log_2|t-\tau|}\big[\chi_{|t-\tau|^{-\beta_1}\gtrsim |x|\gtrsim |t-\tau|^{\beta_1}}([u^{(j-1)}(t, \cdot) - u^{(j-1)}(\tau, \cdot)] \big]
\\&\hspace{5cm}\cdot\tilde{\chi}_{|x|\lesssim |t-\tau|^{-\beta_1}}\big(g_j(\tau)(-\triangle)g_j(\tau)+I\big)^{-1}\tilde{u}\\
\end{split}\end{equation}
The second term on the right can be immediately estimated using \eqref{eqn: 7}, as well as the following bound resulting from standard elliptic estimates\footnote{Indeed, if $v = \big(g_j(\tau)(-\triangle)g_j(\tau)+I\big)^{-1}\tilde{u}$, then we have $\|g_j v\|_{H^1}\lesssim \|\tilde{u}\|_{L^2}$. By radiality, we obtain $|g_jv(r)|\lesssim r^{-1}\|\tilde{u}\|_{L^2}$, from which we get $\|\chi_{r\gtrsim 1}v\|_{L^\infty}\lesssim \|\tilde{u}\|_{L^2}$. Control over $\|\chi_{r\lesssim 1}v\|_{L^\infty}$ follows by standard elliptic estimates on the set $r\lesssim 1$.}:
\[
\|\tilde{\chi}_{|x|\lesssim |t-\tau|^{-\beta_1}}\big(g_j(\tau)(-\triangle)g_j(\tau)+I\big)^{-1}\tilde{u}\|_{L^{\infty}}
\lesssim \|\tilde{u}\|_{L^2}
\]
We thus obtain 
\begin{equation}\nonumber\begin{split}
&\|P_{\geq-\beta_2\log_2|t-\tau|}\big[\chi_{|t-\tau|^{-\beta_1}\gtrsim |x|\gtrsim |t-\tau|^{\beta_1}}([u^{(j-1)}(t, \cdot) - u^{(j-1)}(\tau, \cdot)] \big]
\\&\hspace{5cm}\cdot\tilde{\chi}_{|x|\lesssim |t-\tau|^{-\beta_1}}\big(g_j(\tau)(-\triangle)g_j(\tau)+I\big)^{-1}\tilde{u}\|_{L^2}\\
&\lesssim \|P_{\geq-\beta_2\log_2|t-\tau|}\big[\chi_{|t-\tau|^{-\beta_1}\gtrsim |x|\gtrsim |t-\tau|^{\beta_1}}([u^{(j-1)}(t, \cdot) - u^{(j-1)}(\tau, \cdot)] \big]\|_{L^2}
\\&\hspace{5cm}\cdot\|\tilde{\chi}_{|x|\lesssim |t-\tau|^{-\beta_1}}\big(g_j(\tau)(-\triangle)g_j(\tau)+I\big)^{-1}\tilde{u}\|_{L^\infty}\\
&\lesssim |t-\tau|^{\frac{\beta_2}{2}-\frac{13\beta_1}{4}}\|\tilde{u}\|_{L^2},
\end{split}\end{equation}
which yields the desired bound provided 
\[
\frac{13}{2}\beta_1<\beta_2
\]
We have now reduced to estimating the first term on the right hand side of \eqref{eqn: 8}, for which we use the equation satisfied by $u^{(j-1)}$. 
\\
We start out by estimating 
\[
P_{<-\beta_2\log_2|t-\tau|}\big[\chi_{|t-\tau|^{-\beta_1}\gtrsim |x|\gtrsim |t-\tau|^{\beta_1}}([u^{(j-1)}(t, \cdot) - u^{(j-1)}(\tau, \cdot)] \big]
\]
for which we claim the following 
\begin{lemma}\label{lem:red2}
we have the bound
\[
\|P_{<-\beta_2\log_2|t-\tau|}\big[\chi_{|t-\tau|^{-\beta_1}\gtrsim |x|\gtrsim |t-\tau|^{\beta_1}}([u^{(j-1)}(t, \cdot) - u^{(j-1)}(\tau, \cdot)] \big]\|_{L^\infty}\lesssim |t-\tau|^{1-\frac{7}{2}\beta_2}
\]
\end{lemma}

\begin{proof}(Lemma~\ref{lem:red2}) To see this, we use a bit of Littlewood-Paley calculus: write 
\begin{equation}\nonumber\begin{split}
&P_{<-\beta_2\log_2|t-\tau|}\big[\chi_{|t-\tau|^{-\beta_1}\gtrsim |x|\gtrsim |t-\tau|^{\beta_1}}([u^{(j-1)}(t, \cdot) - u^{(j-1)}(\tau, \cdot)] \big]\\
&=P_{<-\beta_2\log_2|t-\tau|}\big[\chi_{|t-\tau|^{-\beta_1}\gtrsim |x|\gtrsim |t-\tau|^{\beta_1}}P_{<-\beta_2\log_2|t-\tau|+10}[u^{(j-1)}(t, \cdot) - u^{(j-1)}(\tau, \cdot)] \big]\\
&+\sum_{l\geq -\beta_2\log_2|t-\tau|+10}P_{<-\beta_2\log_2|t-\tau|}\big[P_{I_l}\big(\chi_{|t-\tau|^{-\beta_1}\gtrsim |x|\gtrsim |t-\tau|^{\beta_1}}\big)P_{l}[u^{(j-1)}(t, \cdot) - u^{(j-1)}(\tau, \cdot)] \big],\\
\end{split}\end{equation}
where we let $P_{I_l}\eqdef \sum_{a\in [l-5, l+5]}P_a$. Thus we are led to bound the expressions 
\[
P_{<l}[u^{(j-1)}(t, \cdot) - u^{(j-1)}(\tau, \cdot)], l\geq 0,
\]
which we do by expanding 
\begin{equation}\nonumber\begin{split}
&P_{<l}[u^{(j-1)}(t, \cdot) - u^{(j-1)}(\tau, \cdot)]\\
& = \int_{\tau}^{t}P_{<l}\big[(-\triangle)^{-1}(u^{(j-2)})\triangle u^{(j-1)} +\alpha\big(u^{(j-2)}\big)^2\big]\,ds
\end{split}\end{equation}
But then using the fact that the operators $P_{<l}$ are given by smooth convolution kernels with bounded $L^1$-mass (independently of $l$), we easily infer 
\begin{equation}\nonumber\begin{split}
&\big|\int_{\tau}^{t}P_{<l}\big[(-\triangle)^{-1}(u^{(j-2)})\triangle u^{(j-1)} +\alpha\big(u^{(j-2)}\big)^2\big]\,ds\big|\\&
\lesssim |t-\tau| 2^{\frac{7}{2}l}[\|u^{(j-2)}\|_{L^1\cap L^2}^2 +\|u^{(j-1)}\|_{L^1\cap L^2}^2]\lesssim |t-\tau|2^{\frac{7}{2}l} 
\end{split}\end{equation}
Applying this above, we deduce the bound
\begin{equation}\nonumber\begin{split}
&\|P_{<-\beta_2\log_2|t-\tau|}\big[\chi_{|t-\tau|^{-\beta_1}\gtrsim |x|\gtrsim |t-\tau|^{\beta_1}}P_{<-\beta_2\log_2|t-\tau|+10}[u^{(j-1)}(t, \cdot) - u^{(j-1)}(\tau, \cdot)] \big]\|_{L^{\infty}}\\
&\lesssim |t-\tau|^{1-\frac{7}{2}\beta_2}
\end{split}\end{equation}
Furthermore, using that for $2^{l}>> |t-\tau|^{-\beta_1}$, we have for any $N\geq 1$
\[
\|P_{I_l}\big(\chi_{|t-\tau|^{-\beta_1}\gtrsim |x|\gtrsim |t-\tau|^{\beta_1}}\big)\|_{L^\infty}\lesssim_N \big[\frac{1}{|t-\tau|^{\beta_1}2^{l}}\big]^{N},
\]
we can estimate (for $l\geq-\beta_2\log_2|t-\tau|+10$)
\begin{equation}\nonumber\begin{split}
&\|P_{<-\beta_2\log_2|t-\tau|}\big[P_{I_l}\big(\chi_{|t-\tau|^{-\beta_1}\gtrsim |x|\gtrsim |t-\tau|^{\beta_1}}\big)P_{l}[u^{(j-1)}(t, \cdot) - u^{(j-1)}(\tau, \cdot)] \big]\|_{L^\infty}\\
&\lesssim |t-\tau|\big[\frac{1}{|t-\tau|^{\beta_1}2^{l}}\big]^{N} 2^{\frac{7}{2}l},
\end{split}\end{equation}
and summing over $l\geq-\beta_2\log_2|t-\tau|+10$ results in an upper bound (better than) $|t-\tau|^{1-\frac{7}{2}\beta_2}$, which establishes the proof of the above lemma
\end{proof}

We can estimate the first term on the right hand side of \eqref{eqn: 8} by 
\begin{equation}\nonumber\begin{split}
&\|P_{<-\beta_2\log_2|t-\tau|}\big[\chi_{|t-\tau|^{-\beta_1}\gtrsim |x|\gtrsim |t-\tau|^{\beta_1}}([u^{(j-1)}(t, \cdot) - u^{(j-1)}(\tau, \cdot)] \big]
\\&\hspace{5cm}\cdot\tilde{\chi}_{|x|\lesssim |t-\tau|^{-\beta_1}}\big(g_j(\tau)(-\triangle)g_j(\tau)+I\big)^{-1}\tilde{u}\|_{L^2}\\
&\lesssim\|P_{<-\beta_2\log_2|t-\tau|}\big[\chi_{|t-\tau|^{-\beta_1}\gtrsim |x|\gtrsim |t-\tau|^{\beta_1}}([u^{(j-1)}(t, \cdot) - u^{(j-1)}(\tau, \cdot)] \big]\|_{L^\infty}
\\&\hspace{5cm}\cdot\|\tilde{\chi}_{|x|\lesssim |t-\tau|^{-\beta_1}}\big(g_j(\tau)(-\triangle)g_j(\tau)+I\big)^{-1}\tilde{u}\|_{L^2}\\
&\lesssim |t-\tau|^{1-\frac{7}{2}\beta_2}\|\tilde{u}\|_{L^2}
\end{split}\end{equation}
whence we obtain the desired bound if $\beta_2<\frac{2}{7}$. 
\\

The expression \eqref{eq:red9}. Write this term as 
\begin{align*} &g_j(t)^{-2}[\nabla(-\triangle)^{-1}u^{(j-1)}(t, \cdot) - \nabla(-\triangle)^{-1}u^{(j-1)}(\tau, \cdot)]\\&\cdot [\nabla(-\triangle)^{-1}u^{(j-1)}(t, \cdot) + \nabla(-\triangle)^{-1}u^{(j-1)}(\tau, \cdot)]g_j^{-1}(\tau)\cdot\triangle^{-1}\big(g_j^{-1}(\tau)\cdot\big)\circ\Phi\\
&-g_j(t)(\nabla(-\triangle)^{-1}u^{(j-1)}(\tau, \cdot))^2[g_j(\tau)^{-3} - g_j(t)^{-3}]g_j^{-1}(\tau)\cdot\triangle^{-1}\big(g_j^{-1}(\tau)\cdot\big)\circ\Phi\\
\end{align*}
For the first term, use\footnote{This is seen by using the decomposition above \eqref{eq:red4_5}} 
\begin{align*}
&|\nabla(-\triangle)^{-1}u^{(j-1)}(t, x) - \nabla(-\triangle)^{-1}u^{(j-1)}(\tau, x)|\\&\lesssim \min\{|x|^{-\frac{3}{4}}, |x|^{-2}\}(|t-\tau|^{\frac{\beta}{4}} + |t-\tau|^{1-2\beta}),
\end{align*}
as well as 
\[
g_j(t)^{-2}[|\nabla(-\triangle)^{-1}u^{(j-1)}(t, \cdot)| + |\nabla(-\triangle)^{-1}u^{(j-1)}(\tau, \cdot)|]g_j(\tau)^{-1}\lesssim |x|^{-\frac{1}{2}},
\]
\[
\|\big(\triangle^{-1}\big(g_j^{-1}(\tau)\cdot\big)\circ\Phi\big)(\tilde{u})\|_{L^\infty} = \|g_j(\tau)\big(g_j(\tau)(-\triangle)g_j(\tau) + I\big)^{-1}\tilde{u}\|_{L^\infty}\lesssim \|\tilde{u}\|_{L^2},
\]
see the last footnote. We easily obtain the bound 
\begin{align*}
&\|g_j(t)^{-2}[\nabla(-\triangle)^{-1}u^{(j-1)}(t, \cdot) - \nabla(-\triangle)^{-1}u^{(j-1)}(\tau, \cdot)]\\&\cdot [\nabla(-\triangle)^{-1}u^{(j-1)}(t, \cdot) + \nabla(-\triangle)^{-1}u^{(j-1)}(\tau, \cdot)]g_j^{-1}(\tau)\cdot\triangle^{-1}\big(g_j^{-1}(\tau)\cdot\big)\circ\Phi\|\\
&\lesssim (|t-\tau|^{\frac{\beta}{4}} + |t-\tau|^{1-2\beta})
\end{align*}
which is as desired if $\beta<\frac{1}{2}$. For the second term above, we use \eqref{eq:red1}, which allows us to bound 
\begin{align*}
&\big|\big(g_j(t)(\nabla(-\triangle)^{-1}u^{(j-1)}(\tau, \cdot))^2[g_j(\tau)^{-3} - g_j(t)^{-3}]g_j^{-1}(\tau)\cdot\triangle^{-1}\big(g_j^{-1}(\tau)\cdot\big)\circ\Phi\big)(\tilde{u})\big|\\
&\lesssim \min\{|x|^{-1}, |x|^{-2}\}|t-\tau|^{\gamma}\|\tilde{u}\|_{L^2},
\end{align*}
from which we obtain the desired $L^2$-bound 
\begin{align*}
&\|\big(g_j(t)(\nabla(-\triangle)^{-1}u^{(j-1)}(\tau, \cdot))^2[g_j(\tau)^{-3} - g_j(t)^{-3}]g_j^{-1}(\tau)\cdot\triangle^{-1}\big(g_j^{-1}(\tau)\cdot\big)\circ\Phi\big)(\tilde{u})\|_{L^2}\\
&\lesssim |t-\tau|^{\gamma}\|\tilde{u}\|_{L^2}
\end{align*}

We have now estimated the fourth term of \eqref{eqn: 6}, which completes case \eqref{eqn:3c}. 

The estimates in \eqref{eqn:3a}-\eqref{eqn:3c} in turn establish the desired H\"{o}lder estimate \eqref{holderAest} 
for suitable $\gamma>0$. In particular, we have verified all the {\bf{key properties}} which ensure the existence of the {\it{fundamental solution}} $U(t, \tau)$ associated with $A(t)$.

\begin{remark}\label{rem: Holder}  We make the important observation that while the implicit constants in this section depend on the constants $D_j$, one can in fact make them independent of the $D_j$ by choosing the time interval sufficiently small (depending on the $D_j$). To see this, it suffices to shrink the H\"{o}lder exponent a bit. Indeed, assume we have a bound 
\[
\|[A(t) - A(\tau)]A^{-1}(\tau)]\|\leq C(D_1, D_2, D_3)|t-\tau|^{\gamma}
\]
Then by restricting $|t-\tau|\leq c(D_1, D_2, D_3, \gamma)$, we obtain 
\[
\|[A(t) - A(\tau)]A^{-1}(\tau)]\|\leq |t-\tau|^{\frac{\gamma}{2}}
\]

 This has the important consequence that all estimates flowing from Friedman's theory for the parametrix $U(t,s)$ are independent of the $D_j$ as well. 
\end{remark}

\subsection{Obtaining control over $u^{(j)}$}  \label{sec:control}
We re-formulate \eqref{eqn:tildemodelj} as an integral 
equation:
\[
\tilde{u}^{(j)}(t,x) = U(t, 0)\tilde{u}_0
+ 
\int_0^tU(t,s)\left[ - \frac{\partial_s g_j}{g_j}\tilde{u}^{(j)}(s, x) 
+ 
\alpha e^{s} \frac{g_{j-1}^2}{g_j}\left(\tilde{u}^{(j-1)}(s, x)\right)^2\right] ~ ds,
\]
which follows from Duhamel's formula.
Here we have $\tilde{u}_0 = \big[(-\triangle)^{-1}u_0\big]^{-\frac{1}{2}}u_0$. Note that the right hand side depends linearly on $\tilde{u}^{(j)}$, and we will apply a bootstrap argument to control this term.  Alternatively, one could run a secondary iteration to construct $\tilde{u}^{(j)}$, by replacing the $\tilde{u}^{(j)}$ on the left and right by $v^{(k)}, v^{(k-1)}$, respectively, and running an induction on $k$.  We shall prove $L^2$-based estimates for $\tilde{u}^{(j)}$, and then use a direct argument to establish the remaining $L^1, L^{2+}$-bounds.  In the immediately following, we shall derive an a priori bound on 
\begin{equation}
\|\tilde{u}^{(j)}\|_{Z}\eqdef \sup_{t\in [0, T]} \left[\|\tilde{u}^{(j)}(t)\|_{L^2(\threed)} + t^{\frac{1}{2}}\|\la x \ra^{-\frac{1}{2}}\nabla\tilde{u}^{(j)}(t)\|_{L^2(\threed)}\right], 
\label{Znorm}
\end{equation}
assuming inductively the following bound (for $D_5>0$)
\[
\|\tilde{u}^{(k)}\|_{Z}\leq D_5,\quad  k = 0, 1, \ldots, j-1,
\]
in addition to the remaining bounds stated in the Lemma \ref{lem:bounds}. This will partly take care of \eqref{lem:bound1}.

Observe that Sobolev's inequality gives 
\[
t^{\frac{1}{2}}\|\chi_{|x|\lesssim 1}u^{(j)}\|_{L^6(\threed)}\lesssim \|\tilde{u}^{(j)}\|_{Z}.
\]
Now we will estimate each of the terms individually.
\\

{\it{(1) Estimating the expression $U(t, 0)\tilde{u}_0$.}} Observe that due to radiality and monotonicity, as in \eqref{monEST} and \eqref{asymptoticU}, we get 
\[
\| \tilde{u}_0 \|_{L^2(\threed)}
\lesssim
\|\la x \ra^{\frac{1}{2}} u_0\|_{L^2(\threed)}
\lesssim \|u_0\|_{L^2(\threed)} +\|u_0\|_{L^1(\threed)}.
\]
Further, due to the $L^2(\threed)$-boundedness of $U(t, 0)$, we achieve
\[
\|U(t, 0)\tilde{u}_0\|_{L^2}\lesssim \|\tilde{u}_0\|_{L^2}\lesssim \|u_0\|_{L^2} +\|u_0\|_{L^1}.
\]
According to Remark~\ref{rem: Holder}, the implied constant here may be assumed to be independent of the $D_j$ ( at the cost of choosing $T$ small enough). 
Also, due to the operator bound $\|A(t)U(t, 0)\|\lesssim \frac{1}{t}$ and an interpolation type argument, 
we get 
\[
\| \la x \ra^{-\frac{1}{2}}\nabla U(t, 0)\tilde{u}_0\|_{L^2}\lesssim \frac{1}{t^{\frac{1}{2}}}\|\tilde{u}_0\|_{L^2},
\]
whence 
$
\|U(t, 0)\tilde{u}_0\|_{Z}\lesssim \|u_0\|_{L^2(\threed)} + \|u_0\|_{L^1(\threed)}.
$
Indeed, observe that 
\begin{equation}\nonumber\begin{split}
\la g_j\nabla U(t,0)\tilde{u}_0,\,g_j\nabla U(t,0)\tilde{u}_0\ra  =  &\la [g_j, \nabla] U(t,0)\tilde{u}_0, \,g_j\nabla U(t,0)\tilde{u}_0\ra\\& + \la g_j \nabla U(t,0)\tilde{u}_0,\,[g_j, \nabla] U(t,0)\tilde{u}_0\ra\\
&+ \la [\nabla, g_j] U(t,0)\tilde{u}_0,\,[g_j, \nabla] U(t,0)\tilde{u}_0\ra\\
&+\la \nabla g_j U(t,0)\tilde{u}_0, \nabla g_j U(t,0)\tilde{u}_0\ra,
\end{split}\end{equation}
and we easily get 
\[
\big|\la [g_j, \nabla] U(t,0)\tilde{u}_0, \,g_j\nabla U(t,0)\tilde{u}_0\ra\big|\leq \|[g_j, \nabla]\|_{L^\infty}\|U(t,0)\tilde{u}_0\|_{L^2}\|g_j\nabla U(t,0)\tilde{u}_0\|_{L^2}
\]
\[
\big| \la [\nabla, g_j] U(t,0)\tilde{u}_0,\,[g_j, \nabla] U(t,0)\tilde{u}_0\ra\big|\leq \|[\nabla, g_j]\|_{L^\infty}^2\|U(t,0)\tilde{u}_0\|_{L^2}^2
\]
One concludes from the preceding that 
\[
\la g_j\nabla U(t,0)\tilde{u}_0,\,g_j\nabla U(t,0)\tilde{u}_0\ra\lesssim \|[\nabla, g_j]\|_{L^\infty}^2\|\tilde{u}_0\|_{L^2}^2 + \|A(t)U(t,0)\tilde{u}_0\|_{L^2}\|\tilde{u}_0\|_{L^2}
\]
According to Remark~\ref{rem: Holder}, the implied constant in this inequality is independent of the $D_j$;  furthermore, we get 
\[
\|A(t)U(t,0)\tilde{u}_0\|_{L^2}\|\tilde{u}_0\|_{L^2}\lesssim \|\tilde{u}_0\|_{L^2}^2
\]
where the implied constant may be assumed independent of the $D_j$. But then we get 
\[
\|g_j\nabla U(t,0)\tilde{u}_0\|_{L^2}\lesssim t^{-\frac{1}{2}}\|\tilde{u}_0\|_{L^2}[1 + t^{\frac{1}{2}} \|[\nabla, g_j]\|_{L^\infty}]
\]
Finally, we have 
\[
g_j\gtrsim \frac{1}{\la x\ra^{\frac{1}{2}}}
\]
where, using an argument as in \eqref{eqn: conttrick}, we may assume that the implied constant is independent of the $D_j$ on $[0, T]$ for $T$ small enough. We infer 
\[
t^{\frac{1}{2}}\|\la x\ra^{-\frac{1}{2}}\nabla U(t, 0)\tilde{u}_0\|_{L^2}\lesssim \|\tilde{u}_0\|_{L^2},
\]
with implied constant independent of the $D_j$. 
\\

{\it{(2) Estimating the term $\int_0^t U(t, s)\frac{\partial_s g_j}{g_j}\tilde{u}^{(j)}\,ds$}}.  Recalling the definition of $g_j$, we have $\frac{\partial_s g_j}{g_j} = \frac{1}{2}\frac{(-\triangle)^{-1}\partial_s u^{(j-1)}}{(-\triangle)^{-1}u^{(j-1)}}$. 
Then we use the equation \eqref{eqn:modelj} satisfied for $u^{(j-1)}$ to see that what we need to do is estimate 
\[
\int_0^t U(t, s)
\left\{\frac{(-\triangle)^{-1}\big[(-\triangle)^{-1}u^{(j-2)}\triangle u^{(j-1)} + \alpha \big(u^{(j-2)}\big)^2\big]}{(-\triangle)^{-1}u^{(j-1)}}\tilde{u}^{(j)}  \right\}(s, x) ~ ds,
\]
where $U(t, s)$ acts on the entire expression to its right, of course. To estimate the integrand observe that by Newton's formula \eqref{NewtonFormula3D} 
we obtain
\begin{equation}\nonumber\begin{split}
&(-\triangle)^{-1}\big[(-\triangle)^{-1}u^{(j-2)}\triangle u^{(j-1)}\big]\\
&=-\frac{1}{4\pi |x|}\int_{|y|\leq|x|}\nabla(-\triangle)^{-1}u^{(j-2)}\cdot \nabla u^{(j-1)}\,dy\\
&-\int_{|y|>|x|}\nabla\left[\frac{1}{4\pi |y|}(-\triangle)^{-1}u^{(j-2)}\right]\cdot \nabla u^{(j-1)}\,dy.\\
\end{split}\end{equation}
Furthermore we have 
\begin{equation}\nonumber\begin{split}
\|\frac{\la x \ra}{|x|}\int_{|y|\leq|x|}\nabla(-\triangle)^{-1}u^{(j-2)}\cdot \nabla u^{(j-1)}\,dy\|_{L^{\infty}}\lesssim \|\nabla u^{(j-1)}\|_{L^2}[\|u^{(j-2)}\|_{L^2}+\|u^{(j-2)}\|_{L^1}] ,
\end{split}\end{equation}
\begin{equation}\nonumber\begin{split}
&\|\la x \ra \min\{|x|^{\frac{1}{2}+},1\}\int_{|y|>|x|}\nabla\big[\frac{1}{|y|}(-\triangle)^{-1}u^{(j-2)}\big]\cdot \nabla u^{(j-1)}\,dy\|_{L^\infty}\\
&\lesssim \|\nabla u^{(j-1)}\|_{L^2}[\|u^{(j-2)}\|_{L^2} +\|u^{(j-2)}\|_{L^{1}}],
\end{split}\end{equation}
as follows by a straightforward application of H\"{o}lder's inequality. Furthermore, we have 
\[
\|\la x \ra(-\triangle)^{-1}\big(u^{(j-2)}\big)^2\|_{L^{\infty}}\lesssim \|u^{(j-2)}\|_{L^2\cap L_{|x|\lesssim1}^4}^2.
\]
Using the bounds \eqref{boundUinv} for $u^{(j-1)}$ we infer 
\begin{equation}\nonumber\begin{split}
&\|\frac{(-\triangle)^{-1}\big[(-\triangle)^{-1}u^{(j-2)}\triangle u^{(j-1)} + \alpha \big(u^{(j-2)}\big)^2\big]}{(-\triangle)^{-1}u^{(j-1)}}\tilde{u}^{(j)}(s, \cdot)\|_{L^2}\\&\lesssim t^{-\frac{7}{8}-}[\|\tilde{u}^{(j)}\|_{L^2} + t^{\frac{1}{2}}\|\chi_{|x|\lesssim1}\tilde{u}^{(j)}\|_{L^6}]\big[t^{\frac{1}{2}}\|\nabla u^{(j-1)}\|_{L^2}\big]\|u^{(j-2)}\|_{L^2\cap L^1}\\
&+t^{-\frac{3}{4}}[\|u^{(j-2)}\|_{L^2}+t^{\frac{1}{2}}\|\chi_{|x|\lesssim1}u^{(j-2)}\|_{L^6}]^2\|\tilde{u}^{(j)}\|_{L^2}.
\end{split}\end{equation}
We conclude that 
\begin{equation}\nonumber\begin{split}
&\|\int_0^t U(t, s)\frac{(-\triangle)^{-1}\big[(-\triangle)^{-1}u^{(j-2)}\triangle u^{(j-1)} + \alpha \big(u^{(j-2)}\big)^2\big]}{(-\triangle)^{-1}u^{(j-1)}}\tilde{u}^{(j)}(s, \cdot)\,ds\|_{L^2}\\
&\lesssim \big(\int_0^t s^{-(1-)}\,ds\big) \|\tilde{u}^{(j)}\|_{Z}\big[\sum_{k=j-2}^{j-1}\|\tilde{u}^{(k)}\|_{Z\cap L^1}^2\big],\,t\in [0, T],
\end{split}\end{equation}
and further 
\begin{equation}\nonumber\begin{split}
&t^{\frac{1}{2}}\|\la x \ra^{-\frac{1}{2}}\nabla\int_0^t U(t, s)\frac{(-\triangle)^{-1}\big[(-\triangle)^{-1}u^{(j-2)}\triangle u^{(j-1)} + \alpha \big(u^{(j-2)}\big)^2\big]}{(-\triangle)^{-1}u^{(j-1)}}\tilde{u}^{(j)}(s, \cdot)\,ds\|_{L^2}\\
&\lesssim \big(t^{\frac{1}{2}}\int_0^t (t-s)^{-\frac{1}{2}}s^{-(1-)}\,ds\big)\|\tilde{u}^{(j)}\|_{Z}\big[\sum_{k=j-2}^{j-1}\|\tilde{u}^{(k)}\|_{Z\cap L^1}^2\big],\,t\in [0,T].
\end{split}\end{equation}
Choosing $T$ small enough, we then obtain 
\begin{equation}\nonumber\begin{split}
\|\int_0^t U(t, s)\frac{(-\triangle)^{-1}\big[(-\triangle)^{-1}u^{(j-2)}\triangle u^{(j-1)} + \alpha \big(u^{(j-2)}\big)^2\big]}{(-\triangle)^{-1}u^{(j-1)}}\tilde{u}^{(j)}(s, \cdot)\,ds\|_{Z}\ll \|\tilde{u}^{(j)}\|_{Z}.
\end{split}\end{equation}
This is the desired estimate for $\int_0^t U(t, s)\frac{\partial_s g_j}{g_j}\tilde{u}^{(j)}\,ds$.

{\it{(3) Estimating the term $\int_0^t U(t, s) e^{s} \frac{g_{j-1}^2}{g_j}  \big(\tilde{u}^{(j-1)}\big)^2\,ds$}}. 
In this case, we split the integrand into two parts: 
\[
e^{s} \frac{g_{j-1}^2}{g_j}  \big(\tilde{u}^{(j-1)}\big)^2 
= 
\chi_{|x|\lesssim1}
e^{s} \frac{g_{j-1}^2}{g_j}  \big(\tilde{u}^{(j-1)}\big)^2
+  
\chi_{|x|\gtrsim1}
e^{s} \frac{g_{j-1}^2}{g_j}  \big(\tilde{u}^{(j-1)}\big)^2.
\]
Recall that from \eqref{asymptoticU}, we have $e^{s} \frac{g_{j-1}^2}{g_j}\tilde{u}^{(j-1)}\approx e^s g_{j-1}\tilde{u}^{(j-1)} = u^{(j-1)}$. 
Hence recalling radiality and monotonicity, we have  
\[
\|\chi_{|x|\gtrsim1}\frac{g_{j-1}^2}{g_j} e^{2s}\big(\tilde{u}^{(j-1)}\big)^2\|_{L^2}\lesssim \|\tilde{u}^{(j-1)}\|_{L^2}\|u^{(j-1)}\|_{L^1}
\]
From here we easily obtain (here the implied constant may depend on the $D_j$)
\[
\|\int_0^t U(t, s) \chi_{|x|\gtrsim1}\frac{g_{j-1}^2}{g_j} e^{2s}\big(\tilde{u}^{(j-1)}\big)^2\,ds\|_{Z}\lesssim T[\|\tilde{u}^{(j-1)}\|_{Z}+\|u^{(j-1)}\|_{L^1}]^2
\]
and we can make this $\ll D_3+D_5$ (as in the statement of Lemma \ref{lem:bounds}) by picking $T$ small enough.

In the regime of ${|x| \lesssim 1}$, we apply the H\"{o}lder inequality to achieve
\[
\|\chi_{|x|\lesssim1}\frac{g_{j-1}^2}{g_j} e^{2s}\big(\tilde{u}^{(j-1)}(s, \cdot)\big)^2\|_{L^2}\lesssim\|\chi_{|x|\lesssim 1}u^{(j-1)}(s, \cdot)\|_{L^4}^2
\lesssim s^{-\frac{3}{4}}\|\tilde{u}^{(j-1)}\|_{Z}^2.
\]
This also uses Sobolev's embedding, whence we obtain (for $T$ small enough, depending on the $D_j$)
\[
\|\int_{0}^tU(t, s)\chi_{|x|\lesssim1}\frac{g_{j-1}^2}{g_j} e^{2s}\big(\tilde{u}^{(j-1)}(s, \cdot)\big)^2\,ds\|_{Z}\lesssim T^{\frac{1}{4}}\|\tilde{u}^{(j-1)}\|_{Z}^2,\quad t\in [0, T].
\]
This completes the last desired estimate.  

By combining the last three estimates, {\it{(ii1)}} - {\it{(ii3)}}, we obtain 
\[
\|\tilde{u}^{(j)}\|_{Z}\leq \frac{1}{2}\|\tilde{u}^{(j)}\|_{Z} + C_1T^{\frac{1}{4}}[D_3+D_5]^2 +C_2\|u_0\|_{L^1\cap L^2}.
\]
Thus if we pick $D_{3} \gg D_5$ suitably large with respect to $\|u_0\|_{L^1\cap L^2}$ and then $T$ small enough, we recover the bound 
\[
\|\tilde{u}^{(j)}\|_{Z}<D_5,
\]
and via Sobolev's embedding, this of course also gives 
\[
\max_{t\in [0, T]}t^{\frac{1}{2}}\|\chi_{|x|\lesssim 1}u^{(j)}(t, \cdot)\|_{L^6}< D_3.
\]
Thus to complete the deduction of the bounds for the un-differentiated $u^{(j)}$ and hence establishing \eqref{lem:bound1}, we only need to recover the $L^1(\threed)$ and  $L^{2+}(\threed)$-bounds. 

For the $L^1(\threed)$ bounds we revert to the original equation for $u^{(j)}$ as in \eqref{eqn:modelj}.
Integrating over $\threed$ and by parts, we obtain
\[
\partial_t\int_{\threed}u^{(j)}\,dx = -\int_{\threed}u^{(j-1)}u^{(j)}\,dx + \alpha\int_{\threed}\big(u^{(j-1)}\big)^2\,dx,
\]
whence we have 
\[
\int_{\threed}u^{(j)}(t, \cdot)\,dx\leq \int_{\threed}u_0\,dx + \alpha T D_3^2,
\]
from which the desired $L^1(\threed)$ bound follows easily for $T$ small enough. 

Next, we study the a priori bound in ${L^{2+}(\threed)}$. Writing $2+ = 2+\delta$, 
we obtain
\begin{equation}\nonumber\begin{split}
\int_{\threed}\partial_t u^{(j)}\big(u^{(j)}\big)^{1+\delta}\,dy 
= &
\int_{\threed}(-\triangle)^{-1}u^{(j-1)}\nabla \cdot \left[\nabla u^{(j)}\big(u^{(j)}\big)^{1+\delta}\right]\,dy
\\
&
-\left(1+\delta \right) \int_{\threed}(-\triangle)^{-1}u^{(j-1)}|\nabla u^{(j)}|^2\big(u^{(j)}\big)^{\delta}\,dy\\
&+\alpha\int_{\threed}\big(u^{(j-1)}\big)^2\big(u^{(j)}\big)^{1+\delta}
\\
& \leq -\frac{1}{2+\delta}\int_{\threed}\nabla(-\triangle)^{-1}u^{(j-1)}\cdot\nabla \big(u^{(j)}\big)^{2+\delta}\,dy\\
&+\alpha\int_{\threed}\big(u^{(j-1)}\big)^2\big(u^{(j)}\big)^{1+\delta}\leq \alpha\int_{\threed}\big(u^{(j-1)}\big)^2\big(u^{(j)}\big)^{1+\delta}.
\end{split}\end{equation}
We have used
$$
-\int_{\threed}\nabla(-\triangle)^{-1}u^{(j-1)}\cdot\nabla \big(u^{(j)}\big)^{2+\delta}\,dy
=
-\int_{\threed} u^{(j-1)}  \big(u^{(j)}\big)^{2+\delta}\,dy
\le 0.
$$
We obtain  
\begin{equation}\nonumber\begin{split}
\int_{\threed}\big(u^{(j)}\big)^{2+\delta}(t, y)\,dy&\lesssim \int_{\threed}u_0^{2+\delta}(y)\,dy + \int_0^t\int_{\threed}\big(u^{(j-1)}\big)^2\big(u^{(j)}\big)^{1+\delta}\,dyds.
\end{split}\end{equation}
In order to estimate the second integral we split the integrand as
\begin{equation}\nonumber\begin{split}
\big(u^{(j-1)}\big)^2\big(u^{(j)}\big)^{1+\delta}
=
\chi_{|x|\gtrsim 1}\big(u^{(j-1)}\big)^2\big(u^{(j)}\big)^{1+\delta}
 + 
\chi_{|x|\lesssim 1}\big(u^{(j-1)}\big)^2\big(u^{(j)}\big)^{1+\delta}.
\end{split}\end{equation}
For the first term, using the monotonicity as in \eqref{monEST}, 
and interpolation
we get 
\[
\int_0^t\int_{\threed}\chi_{|x|\gtrsim 1}\big(u^{(j-1)}\big)^2\big(u^{(j)}\big)^{1+\delta}\,dyds\leq t\|\chi_{|x|\gtrsim 1}\big(u^{(j-1)}\big)^2\|_{L^\infty}\|u^{(j)}\|_{L^1\cap L^2}^{1+\delta}\ll D_3,
\]
provided we choose $T$ small enough. For the second term, we use H\"{o}lder to obtain 
\[
 \int_0^t\int_{\threed}\chi_{|x|\lesssim 1}\big(u^{(j-1)}\big)^2\big(u^{(j)}\big)^{1+\delta}\,dyds\lesssim \int_0^t\|\chi_{|x|\lesssim 1}u^{(j-1)}\|_{L^{3+\delta}}^2\|\chi_{|x|\lesssim 1}u^{(j)}\|_{L^{3+\delta}}^{1+\delta}\,ds.
 \]
 But again by H\"{o}lder, we have (where $k$ is either $j$ or $j-1$)
 \[
 \|\chi_{|x|\lesssim 1}u^{(k)}\|_{L^{3+\delta}}\lesssim  \|\chi_{|x|\lesssim 1}u^{(k)}\|_{L^2}^{1-\gamma}
  \|\chi_{|x|\lesssim 1}u^{(k)}\|_{L^6}^{\gamma},\,\gamma = 3\left(\frac{1}{2}-\frac{1}{3+\delta}\right).
 \]
Then with the inductive bounds for $u^{(j-1)}$, as well as the established bounds 
for $\tilde{u}^{(j)}$, we achieve
 \[
  \int_0^t\int_{\threed}\chi_{|x|\lesssim 1}\big(u^{(j-1)}\big)^2\big(u^{(j)}\big)^{1+\delta}\,dyds\lesssim\int_0^t s^{-\frac{3(1+\delta)}{4}}\,ds\ll D_3,
\]
provided we choose $T$ small enough. This establishes the $L^{2+}(\threed)$-bound. 

\subsection{Monotonicity of $u^{(j)}$}  \label{sec:mon}
The maximum principle implies that $u^{(j)}>0$ since it solves \eqref{eqn:modelj}. In the rest of this section, we will prove that $u^{(j)}$ is non-increasing.  We apply $x\cdot\nabla_x = r\partial_r$, $r=|x|$, to the equation \eqref{eqn:modelj} to obtain 
\begin{gather}\nonumber
(-\triangle)^{-1}u^{(j-1)}\left\{(x\cdot\nabla_x)\triangle u^{(j)}\right\} 
+ 
\left\{ (x\cdot\nabla_x)(-\triangle)^{-1}u^{(j-1)}\right\}
\triangle u^{(j)} - \partial_t (x\cdot\nabla_x)u^{(j)}\\
 = -2\left\{ \alpha(x\cdot\nabla_x)u^{(j-1)}\right\} u^{(j-1)}.
 \notag
\end{gather}
Then we look at the commutator, $[A,B]=AB - BA$, as follows
\begin{equation}\nonumber\begin{split}
\triangle u &= \frac{1}{r}\partial_r(r\partial_r u) + \frac{1}{r}\partial_r u,
\\
[x\cdot\nabla_x, \triangle] u &=  -\frac{2}{r}\partial_r(r\partial_r u) - \frac{2}{r}\partial_r u.
\end{split}\end{equation}
Furthermore, due to radiality of $u^{(j-1)}$, we have 
from \eqref{NewtonFormula3D} that
\[
(x\cdot\nabla_x)(-\triangle)^{-1}u^{(j-1)} = -\frac{1}{4\pi r}\int_{|y|\leq|x|}u^{(j-1)}(t,y) dy.
\]
We collect these last few calculations to obtain 
\begin{equation}\nonumber\begin{split}
&(-\triangle)^{-1}u^{(j-1)}\triangle \left\{(x\cdot\nabla_x)u^{(j)}\right\} - (-\triangle)^{-1}u^{(j-1)}
\left\{\frac{2}{r}\partial_r(r\partial_r u^{(j)}) + \frac{2}{r}\partial_r u^{(j)}\right\}
\\
&
 - \left(\frac{1}{4\pi r}\int_{|y|\leq|x|}u^{(j-1)}dy\right)
 \left\{\frac{1}{r}\partial_r(r\partial_r u^{(j)}) + \frac{1}{r}\partial_r u^{(j)}\right\} 
 - \partial_t (x\cdot\nabla_x)u^{(j)}
 \\
& = -2\left\{ \alpha(x\cdot\nabla_x)u^{(j-1)}\right\}u^{(j-1)}.
\end{split}\end{equation}
Here the key feature is that the coefficient of $z = (r\partial_r)u^{(j)}$ in the above is strictly negative, while the right hand side is non-negative by assumption. Now by the maximal principle, the solution of  
\begin{equation}\nonumber\begin{split}
&(-\triangle)^{-1}u^{(j-1)}\triangle z -  (-\triangle)^{-1}u^{(j-1)} 
\left\{\frac{2}{r}\partial_r z + \frac{2}{r^{2}}z\right\} 
\\&-  \big(\frac{1}{4\pi r}\int_{|y|\leq|x|}u^{(j-1)}dy\big)
 \left\{\frac{1}{r}\partial_r z + \frac{1}{r^2}z\right\} 
  -\partial_t z
  \\
 &= -2\alpha  \left\{ (x\cdot\nabla_x)u^{(j-1)} \right\}  u^{(j-1)}, 
\end{split}\end{equation}
on $B_R = \{x\in \threed ~| ~ |x|\leq R\}$ with initial data
$$
z(0, x) =(r\partial_r)\big[\chi_{|x|\leq \frac{R}{2}}(\phi_R* u_0)\big],
$$
and boundary conditions $z(t, \cdot)|_{\partial B_R} = 0$, 
 where $\phi_R$ is a standard mollifier with $\lim_{R\to\infty}\phi_R*u_0 = u_0$
 and $\chi_{|x|\leq \frac{R}{2}}$ is a smooth truncation of the indicated region, 
 cannot attain a positive maximum. 
Hence the solution $\tilde{u}^{(j), R}$ of the problem 
\[
\partial_t\tilde{u}^{(j), R} = (-\triangle)^{-1}u^{(j-1)}\triangle \tilde{u}^{(j), R} + \alpha\big(u^{(j-1)}\big)^2,
\quad 
\tilde{u}^{(j), R}(0, \cdot) = \chi_{|x|\leq \frac{R}{2}}(\phi_R* u_0),
\]
with vanishing boundary values on $\partial B_R$ is non-increasing. 

A simple limiting argument, using analogous bounds to those obtained in the preceding subsections, then shows that $u^{(j)}$, being the limit of a sequence of non-increasing functions, is itself non-increasing. 

\subsection{Controlling the elliptic operator: lower bound for $(-\triangle)^{-1}u^{(j)}$.}\label{sec:ELLcontrol}
In this section we will prove the lower bound from
\eqref{boundUinv} for $u = u^{(j)}$ 
as in
\begin{equation}\label{eqn:controlbelow}
(-\triangle)^{-1}u^{(j)} > \frac{D_2}{\la x \ra}.
\end{equation}
We begin with the assumption $A_0\eqdef\int_{r_0<|x|<r_0^{-1}}u_0(x)\,dx>0$.   Now choose a smooth cutoff function 
$
\chi(r)\in C_0^{\infty}(\R_{>0})
$ 
which satisfies ($0<r_0<1$):
\begin{equation}
\notag
\chi(r)
=
\left\{
\begin{array}{cc}
0, & r\leq \frac{r_0}{2}, \quad \text{or} \quad r\geq 2r_0^{-1},
\\
1, & r \in [r_0, r_0^{-1}].
\end{array}
\right.
\end{equation}
We may furthermore suppose that $\chi$ satisfies 
\begin{equation}
\notag
\chi(r)
=
\left\{
\begin{array}{cc}
e^{-\frac{1}{r-\frac{r_0}{2}}}, & r \in [\frac{r_0}{2}, \frac{3}{4}r_0],
\\
e^{-\frac{1}{2r_0^{-1}-r}}, & r \in [\frac{5}{4}r_0^{-1}, 2r_0^{-1}].
\end{array}
\right.
\end{equation}
We abuse notation to write $\chi(|y|) = \chi(y)$ for $y\in\threed$.
Then consider the function 
\[
f(t)\eqdef \int_{\threed}~\chi(y)~ u^{(j)}(t, y)\,dy
=
\int_{\frac{r_0}{2} < |y| < \frac{2}{r_0}  }~\chi(y)~ u^{(j)}(t, y)\,dy.
\]
We compute 
\begin{equation}\nonumber\begin{split}
f'(t) = &\int_{\threed}\chi(y)\big[(-\triangle)^{-1}u^{(j-1)}\triangle u^{(j)} +\alpha \big(u^{(j-1)}\big)^2\big]\,dy\\
 =& \int_{\threed}\chi(y)\big[\alpha \big(u^{(j-1)}\big)^2-u^{(j-1)}u^{(j)}\big]\,dy
 \\
&+ 2\int_{\threed}\nabla_y\chi(y)\cdot \nabla_y(-\triangle)^{-1}u^{(j-1)}u^{(j)}\,dy\\
&+ \int_{\threed}\triangle_y\chi(y)(-\triangle)^{-1}u^{(j-1)}u^{(j)}\,dy.
\end{split}\end{equation}
By our choice of $\chi$, and the positivity of $u^{(j-1)}$, $u^{(j)}$, 
we have 
\[
\triangle_y\chi(y)(-\triangle)^{-1}u^{(j-1)}u^{(j)}+2\nabla_y\chi(y)\cdot \nabla_y(-\triangle)^{-1}u^{(j-1)}u^{(j)}>0,
\]
for $\big||y|-\frac{r_0}{2}\big|\ll1$, $\big||y|-\frac{2}{r_0}\big|\ll 1$; indeed, use that $\triangle_y = \partial_r^2 + \frac{2}{r}\partial_r$ in the radial context and hence for $\big||y| - \frac{r_0}{2}\big|\ll 1$ we have $\triangle_y\chi(y) = \frac{1}{(r-\frac{r_0}{2})^4}e^{-\frac{1}{r-\frac{r_0}{2}}} + O(\frac{1}{(r-\frac{r_0}{2})^3}e^{-\frac{1}{r-\frac{r_0}{2}}})$, $\nabla_y\chi(y) = O(\frac{1}{(r-\frac{r_0}{2})^2}e^{-\frac{1}{r-\frac{r_0}{2}}})$.   Further, we have 
\[
\big|\triangle_y\chi(y)\big| +\big|\nabla_y\chi(y)|\lesssim_{r_0, \delta}\chi(y),
\]
for $|y|\in [\frac{r_0}{2}+\delta, \frac{2}{r_0}-\delta]$ for some small $\delta >0$.

Using the radiality and monotonicity of $u^{(j)}$ and $u^{(j-1)}$ as in \eqref{monEST}, we then conclude that 
\[
f'(t) \geq -C(r_0, D_3)f(t),\quad f(0)\geq A_0,
\]
whence we get 
\[
f(t)\geq e^{-C(r_0, D_3)T}A_0,\quad t\in [0, T].
\]
In particular, from \eqref{NewtonFormula3D} we get 
\[
(-\triangle)^{-1}u^{(j)}(t, x)\geq \frac {1}{4\pi |x|}e^{-C(r_0, D_3)T}A_0,\,|x|>2r_0^{-1},\,t\in [0, T].
\]
Also, by monotonicity of $(-\triangle)^{-1}u^{(j)}$ with respect to $|x|$, 
we  get 
\[
(-\triangle)^{-1}u^{(j)}(t, x)\geq \frac {1}{4\pi \la x \ra}\frac{r_0}{2}e^{-C(r_0, D_3)T}A_0,\quad |x|\leq 2r_0^{-1},\,t\in [0,T].
\]
Note that the factor $\la x \ra$ in the denominator is not needed in this last lower bound.
We can thus recover the bound \eqref{eqn:controlbelow} provided we have 
\[
D_2< \frac {1}{4\pi}\frac{r_0}{2}e^{-C(r_0, D_3)T}A_0.
\]
This concludes our proof of \eqref{eqn:controlbelow}.

\begin{remark}\label{rem: lowerbound} The preceding proof reveals that in fact $D_2$ can be chosen to be depend only on $r_0$, $\int_{r_0<|x|<r_0^{-1}}u(x)\,dx$, $T$, due to the monotonicity properties of $u$. 
\end{remark}

\subsection{Controlling the elliptic operator: upper bound for $(-\triangle)^{-1}u^{(j)}$.}\label{sec:UpperELLcontrol} Here we prove the bound $(-\triangle)^{-1}u^{(j)}< D_1$. Note that formula \eqref{NewtonFormula3D} implies that 
\[
(-\triangle)^{-1}u^{(j)}\lesssim \|u^{(j)}\|_{L^2\cap L^1}<D_3\ll D_1,
\]
provided we pick $D_1\gg D_3$.

\subsection{Higher derivative bounds} \label{sec:higherDb}
Here we prove the bounds on $\nabla^{\alpha}u^{(j)}$, $0\leq|\alpha|\leq 2$, claimed in Lemma \ref{lem:bounds}, and in particular \eqref{lem:bound2}. Our point of departure is again the integral identity
\[
\tilde{u}^{(j)} = U(t, 0)\tilde{u}_0 + \int_0^tU(t,s)\big[ - \frac{\partial_s g_j}{g_j}\tilde{u}^{j}(s, \cdot) + \alpha \frac{g_{j-1}^2}{g_j} e^{s}\big(\tilde{u}^{(j-1)}(s, \cdot)\big)^2\big]\,ds,
\]
whence we get 
\begin{equation}\label{eqn: Deltau}
\triangle\tilde{u}^{(j)} = \triangle U(t, 0)\tilde{u}_0 + \int_0^t \triangle U(t,s)\big[ - \frac{\partial_s g_j}{g_j}\tilde{u}^{j})(s, \cdot) + \alpha \frac{g_{j-1}^2}{g_j} e^{s}\big(\tilde{u}^{(j-1)}(s, \cdot)\big)^2\big]\,ds,
\end{equation}
We start with the linear term $v\eqdef \triangle U(t, 0)\tilde{u}_0 $. Note that $v$ satisfies the equation 
\[
\partial_t v = \triangle\big(g_j\triangle g_j\big) u
\]
where we put $u = U(t, 0)\tilde{u}_0$. Expanding this out, we obtain 
\[
\partial_t v = g_j\triangle (g_j v) + (\triangle g_j)\triangle (g_j u)+ 2\nabla g_j\cdot\nabla\triangle(g_j u)
+2g_j\triangle(\nabla g_j\cdot\nabla u) + g_j\triangle(\triangle g_j u),
\]
whence 
\begin{equation}\nonumber\begin{split}
v(t, \cdot)
 &= U(t, 0)v_0\\& + \int_0^t U(t, s)\big[(\triangle g_j)\triangle g_j u+2\nabla g_j\cdot\nabla\triangle(g_j u)
+2g_j\triangle(\nabla g_j\cdot\nabla u) + g_j\triangle(\triangle g_j u)\big]\,ds
\end{split}\end{equation}
We shall use the above to derive an a priori bound on $\|v\|_{Z}$, where $\|.\|_{Z}$ is as above in \eqref{Znorm}. Note that we have a schematic identity of the form 
\begin{equation}\nonumber\begin{split}
&(\triangle g_j)\triangle (g_j u)+2\nabla g_j\cdot\nabla\triangle(g_j u)
+2g_j\triangle(\nabla g_j\cdot\nabla u) + g_j\triangle(\triangle g_j u)\\
&=\sum_{|\alpha_1|\leq 2, |\alpha_3|\leq 3 \,\sum_{j=1}^3|\alpha_j|=4}\nabla^{\alpha_1} g_j\nabla^{\alpha_2}g_j\nabla^{\alpha_3}u 
\end{split}\end{equation}
These estimates are fairly tedious but essentially straightforward. We treat here the extreme cases $|\alpha_3|=0$, $|\alpha_3|=3$. 
\\

{\it{$|\alpha_3|=0$.}} We can write this case as 
\[
\sum_{|\alpha_1|\leq 2,\,|\alpha_1|+|\alpha_2|=4}\nabla^{\alpha_1}g_j \nabla^{\alpha_2}g_2 u
\]
We treat the cases $|\alpha_1|=2$ and $|\alpha_1|=0$, the remaining one being analogous. In the former, we get (another schematic identity)
\begin{equation}\nonumber\begin{split}
\nabla^{\alpha_1}g_j\nabla^{\alpha_2}g_j u =& \nabla^{\alpha_1}(-\triangle)^{-1}u^{(j-1)} \nabla^{\alpha_2}(-\triangle)^{-1}u^{(j-1)} g_j^{-2} u\\
&+\big(\nabla(-\triangle)^{-1}u^{(j-1)}\big)^2g_j^{-3}\big(\nabla(-\triangle)^{-1}u^{(j-1)}\big)^2g_j^{-3}u \\
&+\ldots
\end{split}\end{equation}
where we have again omitted similar terms. For the first term, use 
\[
\|\chi_{|x|\gtrsim 1}\la x \ra^{\frac{1}{2}}\nabla^{\alpha_1}(-\triangle)^{-1}u^{(j-1)}\|_{L^{\infty}}\lesssim 1,
\]
whence we get 
\[
\|\chi_{|x|\gtrsim 1}\nabla^{\alpha_1}(-\triangle)^{-1}u^{(j-1)} \nabla^{\alpha_2}(-\triangle)^{-1}u^{(j-1)} g_j^{-2} u\|_{L^2}\lesssim \|u\|_{L^2},
\]
while we also have 
\[
\|\chi_{|x|\lesssim1}\la x \ra^{\frac{1}{2}}\nabla^{\alpha_1}(-\triangle)^{-1}u^{(j-1)}(t, \cdot)\|_{L^{4}}\lesssim t^{-\frac{3}{8}};
\]
indeed, recall that $\|u^{(j-1)}\|_{Z}\leq D_5$.\\ 
Next, by Sobolev's embedding we have 
\[
\|u\|_{L^{\infty}}\lesssim \|v\|_{L^2} + \|u\|_{L^2},
\]
whence we get 
\begin{equation}\nonumber\begin{split}
&\|\chi_{|x|\lesssim 1}\nabla^{\alpha_1}(-\triangle)^{-1}u^{(j-1)} \nabla^{\alpha_2}(-\triangle)^{-1}u^{(j-1)} g_j^{-2} u\|_{L^2}\\&
\lesssim \prod_{k=1,2}\|\chi_{|x|\lesssim 1}\nabla^{\alpha_k}(-\triangle)^{-1}u^{(j-1)}\|_{L^4}^2(\|v\|_{L^2}+\|u\|_{L^2})\lesssim t^{-\frac{3}{4}}(\|v\|_{L^2} + \|u\|_{L^2})
\end{split}\end{equation}
Next, for the term 
\[
\big(\nabla(-\triangle)^{-1}u^{(j-1)}\big)^2g_j^{-3}\big(\nabla(-\triangle)^{-1}u^{(j-1)}\big)^2g_j^{-3}u,
\]
use 
\[
\|\big(\nabla(-\triangle)^{-1}u^{(j-1)}\big)^2g_j^{-3}\|_{L^{\infty}}\lesssim \|u^{(j-1)}\|_{L_{|x|\lesssim1}^4\cap L^1}^2\lesssim t^{-\frac{3}{4}}
\]
\[
\|\big(\nabla(-\triangle)^{-1}u^{(j-1)}\big)^2g_j^{-3}\|_{L^{2}}\lesssim \|u^{(j-1)}\|_{L^1\cap L^2}^2
\]
whence 
\[
\|\big(\nabla(-\triangle)^{-1}u^{(j-1)}\big)^2g_j^{-3}\big(\nabla(-\triangle)^{-1}u^{(j-1)}\big)^2g_j^{-3}u\|_{L^2}\lesssim t^{-\frac{3}{4}}(\|v\|_{L^2} + \|u\|_{L^2})
\]
Next, if $|\alpha_1| = 0$, i.e. we have a term $g_j (\nabla^{\alpha_2}g_j) u$ with $|\alpha_2|=4$, we can expand schematically
\[
\nabla^{\alpha_2}g_j = g_j^{-1}\nabla^{\alpha_2}(-\triangle)^{-1}u^{(j-1)}+\ldots+g_j^{-7}\big(\nabla(-\triangle)^{-1}u^{(j-1)}\big)^4
\]
where we omit 'intermediate' terms. Then we estimate the contribution of the first term by 
\begin{equation}\nonumber\begin{split}
\|\chi_{|x|\lesssim 1}g_j  g_j^{-1}\nabla^{\alpha_2}(-\triangle)^{-1}u^{(j-1)}u\|_{L^2}&\lesssim \|\chi_{|x|\lesssim 1}\nabla^{\alpha_2}(-\triangle)^{-1}u^{(j-1)}\|_{L^4}\|\chi_{|x|\lesssim 1}u\|_{4}\\&\lesssim t^{-\frac{3}{4}}[\triangle u^{(j-1)}\|_{Z}+\|u^{(j-1)}\|_{Z}]\|u\|_{Z}
\end{split}\end{equation}
where we are invoking the bounds 
\[
 \|\chi_{|x|\lesssim 1}\nabla^{\alpha_2}(-\triangle)^{-1}u^{(j-1)}\|_{L^4}\lesssim t^{-\frac{3}{8}}[\|\triangle u^{(j-1)}\|_{Z}+\|u^{(j-1)}\|_{Z}]
 \]
 \begin{equation}\nonumber\begin{split}
 \|\chi_{|x|\gtrsim 1}g_j  g_j^{-1}\nabla^{\alpha_2}(-\triangle)^{-1}u^{(j-1)}u\|_{L^2}&\lesssim \|\nabla^{\alpha_2}(-\triangle)^{-1}u^{(j-1)}\|_{L^2}\|\chi_{|x|\gtrsim 1}u\|_{L^\infty}\\&\lesssim \|\triangle u^{(j-1)}\|_{Z}+\|u^{(j-1)}\|_{Z}
 \end{split}\end{equation}
 For the second term above, we have 
 \[
 \chi_{|x|\lesssim 1}g_j^{-7}\big(\nabla(-\triangle)^{-1}u^{(j-1)}\big)^4\lesssim  \chi_{|x|\lesssim 1}|x|^{-\frac{5}{4}}\|u^{(j-1)}\|_{L^2}^3\|\chi_{|x|\lesssim1}u^{(j-1)}\|_{L^4},
\]
whence we obtain 
\begin{equation}\nonumber\begin{split}
\| \chi_{|x|\lesssim 1}g_j\cdot g_j^{-7}\big(\nabla(-\triangle)^{-1}u^{(j-1)}\big)^4\cdot u\|_{L^2}
\lesssim  &\|\chi_{|x|\lesssim 1}g_j^{-7}\big(\nabla(-\triangle)^{-1}u^{(j-1)}\big)^4\|_{L^{2}}\|u\|_{L^{\infty}}\\
&\lesssim t^{-\frac{3}{8}}\|v\|_{L^2}
\end{split}\end{equation}
The contribution in  the region $|x|\gtrsim 1$ is again much simpler due to radiality. 
When inserting the preceding estimates into the Duhamel formula, we can summarize these estimates by 
\begin{equation}\nonumber\begin{split}
&\|\int_0^t U(t, s)\big[\nabla^{\alpha_1}g_j\nabla^{\alpha_2}g_j u\big]\,ds\|_{Z}
\\&\lesssim \int_0^t(1+t^{\frac{1}{2}}(t-s)^{-\frac{1}{2}})s^{-\frac{3}{4}}[\|v(s, \cdot)\|_{Z}+\|\triangle \tilde{u}^{(j-1)}\|_{Z} +1]\,ds\\
&\lesssim T^{\frac{1}{4}}[\|v\|_{Z}+\|\triangle \tilde{u}^{(j-1)}\|_{Z} +1]
\end{split}\end{equation}
where the implied constant does not depend on $\|v\|_Z$ or $\|\triangle \tilde{u}^{(k)}\|_{Z}$, $k\leq j-1$, but only on the a priori bounds derived in Subsection \ref{sec:control}. 
\\

$|\alpha_3| = 3$.  We next treat the contribution of the expressions $g_j\nabla g_j \nabla^{\alpha_3}u$ with $|\alpha_3| =3$. This is schematically the same as $\nabla(-\triangle)^{-1}u^{(j-1)}\nabla^{\alpha_3}u$. We write
\[
\nabla(-\triangle)^{-1}u^{(j-1)}\nabla^{\alpha_3}u
=\chi_{|x|\gtrsim 1}\nabla(-\triangle)^{-1}u^{(j-1)}\nabla^{\alpha_3}u + \chi_{|x|\lesssim1}\nabla(-\triangle)^{-1}u^{(j-1)}\nabla^{\alpha_3}u
\]
For the first term, we can estimate 
\[
\|\chi_{|x|\gtrsim 1}\nabla(-\triangle)^{-1}u^{(j-1)}\nabla^{\alpha_3}u\|_{L^2}\lesssim \|\la x \ra^{-\frac{1}{2}}\nabla\triangle u\|_{L^2}\lesssim t^{-\frac{1}{2}}\|v\|_{Z},
\]
while for the second term, we have 
\[
\|\chi_{|x|\lesssim1}\nabla(-\triangle)^{-1}u^{(j-1)}\nabla^{\alpha_3}u\|_{L^2}\lesssim \|\chi_{|x|\lesssim1}\nabla(-\triangle)^{-1}u^{(j-1)}\|_{L^\infty}\|\chi_{|x|\lesssim1}\nabla^{\alpha_3}u\|_{L^2}\lesssim t^{-\frac{3}{4}}\|v\|_{Z}.
\]
Summarizing the preceding estimates, we have proved that 
\begin{equation}\nonumber\begin{split}
&\|\int_0^tU(t, s)\sum_{|\alpha_1|\leq 2, |\alpha_3|\leq 3 \,\sum_{j=1}^3|\alpha_j|=4}\nabla^{\alpha_1} g_j\nabla^{\alpha_2}g_j\nabla^{\alpha_3}u\,ds\|_Z
\\&\lesssim \int_0^t (1+t^{\frac{1}{2}}(t-s)^{-\frac{1}{2}})s^{-\frac{3}{4}}[\|v(s, \cdot)\|_{Z}+\|\triangle\tilde{u}^{(j-1)}\|_{Z} + 1]\,ds
\end{split}\end{equation}
whence recalling the equation for $Z$ stated further above, we get 
\[
\|v\|_{Z} \lesssim \|v_0\|_{L^2} +T^{\frac{1}{4}}[\|v\|_{Z}+\|\triangle\tilde{u}^{(j-1)}\|_{Z} + 1]
\]
from which we get $\|v\|_{Z}\lesssim \|v_0\|_{L^2}+T^{\frac{1}{4}}\|\triangle\tilde{u}^{(j-1)}\|_{Z}+1$; here the same remark applies about the implied constant as before. 
In particular, recalling the equation \eqref{eqn: Deltau} for $\triangle\tilde{u}^{(j)}$, we have 
\[
\|\triangle U(t, 0)\tilde{u}_0\|_Z\lesssim \|\tilde{v}_0\|_{L^2}+[T^{\frac{1}{4}}\|\triangle\tilde{u}^{(j-1)}\|_{Z}+1]\|\tilde{u}_0\|_{L^2}
\]
Next, consider the integral term in \eqref{eqn: Deltau}. Thanks to the immediately preceding, we have 
\begin{equation}\nonumber\begin{split}
&\|\int_0^t \triangle U(t,s)\big[ - \frac{\partial_s g_j}{g_j}\tilde{u}^{j})(s, \cdot) + \alpha \frac{g_{j-1}^2}{g_j} e^{2s}\big(\tilde{u}^{(j-1)}(s, \cdot)\big)^2\big]\,ds\|_{Z}\\
&\lesssim \int_0^t \|\triangle \big[ - \frac{\partial_s g_j}{g_j}\tilde{u}^{j})(s, \cdot) + \alpha \frac{g_{j-1}^2}{g_j} e^{2s}\big(\tilde{u}^{(j-1)}(s, \cdot)\big)^2\big]\|_{L^2}\,ds\\
&+[T^{\frac{1}{4}}\|\triangle\tilde{u}^{(j-1)}\|_{Z}+1]\int_0^t \|\big[ - \frac{\partial_s g_j}{g_j}\tilde{u}^{j})(s, \cdot) + \alpha \frac{g_{j-1}^2}{g_j} e^{2s}\big(\tilde{u}^{(j-1)}(s, \cdot)\big)^2\big]\|_{L^2}\,ds.
\end{split}\end{equation}
Here the second expression on the right is of course treated like in Subsection \ref{sec:control}, and so it suffices to consider the first expression on the right. 
We treat a number of different contributions separately: 
\\

{\it{Contribution of $\triangle\big(\frac{\partial_s g_j}{g_j}\tilde{u}^{j})\big) = \frac{\partial_s g_j}{g_j}\triangle \tilde{u}^{(j)} + \triangle\big[\frac{\partial_s g_j}{g_j}\big]\tilde{u}^{(j)} + 2\nabla\big[\frac{\partial_s g_j}{g_j}\big]\cdot\nabla\tilde{u}^{(j)} $.}} 
For the first term on the right, use the estimates in case {\it{(ii2)}} in Subsection \ref{sec:control} to conclude
\begin{equation}\nonumber\begin{split}
&\|\frac{\partial_s g_j}{g_j}\triangle \tilde{u}^{(j)}(s, \cdot)\|_{L^2}\\&\lesssim t^{-\frac{7}{8}}[\|\triangle\tilde{u}^{(j)}\|_{L^2} + t^{\frac{1}{2}}\|\chi_{|x|\lesssim1}\triangle\tilde{u}^{(j)}\|_{L^6}]\big[t^{\frac{1}{2}}\|\nabla u^{(j-1)}\|_{L^2}\big]\|u^{(j-2)}\|_{L^2\cap L^1}\\
&+t^{-\frac{3}{4}}[\|u^{(j-2)}\|_{L^2}+t^{\frac{1}{2}}\|\chi_{|x|\lesssim1}u^{(j-2)}\|_{L^6}]^2\|\triangle\tilde{u}^{(j)}\|_{L^2}
\end{split}\end{equation}
In particular, we get (for suitable $\nu>0$)
\[
\max_{t\in [0, T]}\int_0^t \|\frac{\partial_s g_j}{g_j}\triangle \tilde{u}^{(j)}(s, \cdot)\|_{L^2}\,ds\lesssim T^{\nu}\|\triangle \tilde{u}^{(j)}\|_{Z}
\]
where the implied constant only depends on the a priori bounds on $u^{(k)}$, $k\leq j-1$, derived in Subsection \ref{sec:control}.

Next, consider the contribution of $\triangle\big[\frac{\partial_s g_j}{g_j}\big]\tilde{u}^{(j)}$. This is again tedious but requires no new ideas to estimate: decompose 
\begin{equation}\nonumber\begin{split}
\triangle\big[\frac{\partial_s g_j}{g_j}\big]\tilde{u}^{(j)}&= \frac{(-\triangle)^{-1}u^{(j-2)}\triangle u^{(j-1)} + \alpha \big(u^{(j-1)}\big)^2}{(-\triangle)^{-1}u^{(j-1)}}\tilde{u}^{(j)}\\
&+2\frac{\nabla(-\triangle)^{-1}\big[(-\triangle)^{-1}u^{(j-2)}\triangle u^{(j-1)} + \alpha \big(u^{(j-1)}\big)^2\big]}{\big[(-\triangle)^{-1}u^{(j-1)}\big]^2}\cdot\nabla(-\triangle)^{-1}u^{(j-1)}\tilde{u}^{(j)}\\
&+(-\triangle)^{-1}\big[(-\triangle)^{-1}u^{(j-2)}\triangle u^{(j-1)} + \alpha \big(u^{(j-1)}\big)^2\big]\triangle\big[\frac{1}{(-\triangle)^{-1}u^{(j-1)}}\big]\tilde{u}^{(j)}
\end{split}\end{equation}
We estimate the first expression on the right, the second and third being more of the same. For the first, write it as 
\begin{equation}\nonumber\begin{split}
&\frac{(-\triangle)^{-1}u^{(j-2)}\triangle u^{(j-1)} + \alpha \big(u^{(j-1)}\big)^2}{(-\triangle)^{-1}u^{(j-1)}}\tilde{u}^{(j)}\\
&=\chi_{|x|\lesssim1}\frac{(-\triangle)^{-1}u^{(j-2)}\triangle u^{(j-1)} + \alpha \big(u^{(j-1)}\big)^2}{(-\triangle)^{-1}u^{(j-1)}}\tilde{u}^{(j)}\\
&+\chi_{|x|\gtrsim 1}\frac{(-\triangle)^{-1}u^{(j-2)}\triangle u^{(j-1)} + \alpha \big(u^{(j-1)}\big)^2}{(-\triangle)^{-1}u^{(j-1)}}\tilde{u}^{(j)}\\
\end{split}\end{equation}
Estimate the first expression on the right via
\begin{equation}\nonumber\begin{split}
&\|\chi_{|x|\lesssim1}\frac{(-\triangle)^{-1}u^{(j-2)}\triangle u^{(j-1)} + \alpha \big(u^{(j-1)}\big)^2}{(-\triangle)^{-1}u^{(j-1)}}\tilde{u}^{(j)}\|\\
&\lesssim \|\triangle u^{(j-1)}\|_{L^4}\|\tilde{u}^{(j)}\|_{L^4}\|u^{(j-2)}\|_{L^2}+\|u^{(j-1)}\|_{L^{\infty}}\|u^{(j-1)}\|_{L^4}\|\tilde{u}^{(j)}\|_{L^4}\\
&\lesssim t^{-\frac{3}{4}}\|\triangle u^{(j-1)}\|_{Z} 
\end{split}\end{equation}
where the absolute constant only depends on the bounds established in Subsection \ref{sec:control}. On the other hand, we can estimate 
\begin{equation}\nonumber\begin{split}
&\|\chi_{|x|\gtrsim 1}\frac{(-\triangle)^{-1}u^{(j-2)}\triangle u^{(j-1)} + \alpha \big(u^{(j-1)}\big)^2}{(-\triangle)^{-1}u^{(j-1)}}\tilde{u}^{(j)}\|_{L^2}\\&\lesssim \|u^{(j-2)}\|_{L^1}\|\triangle u^{(j-1)}\|_{L^2}\|\chi_{|x|\gtrsim 1}\tilde{u}^{(j)}\|_{L^{\infty}}+\|\chi_{|x|\gtrsim 1}\la x \ra^{\frac{1}{2}}u^{(j-1)}\|_{L^\infty}^2\|\tilde{u}^{(j)}\|_{L^2}\\
&\lesssim \|\triangle u^{(j-1)}\|_{L^2}+1
\end{split}\end{equation}

Finally, we consider the contribution of the third term above, $2\nabla\big[\frac{\partial_s g_j}{g_j}\big]\cdot\nabla\tilde{u}^{(j)}$. Write it as 
 \begin{equation}\nonumber\begin{split}
 &\frac{\nabla(-\triangle)^{-1}\big[(-\triangle)^{-1}u^{(j-2)}\triangle u^{(j-1)} + \alpha \big(u^{(j-1)}\big)^2\big]}{(-\triangle)^{-1}u^{(j-1)}}\cdot\nabla\tilde{u}^{(j)}\\
 &+\frac{(-\triangle)^{-1}\big[(-\triangle)^{-1}u^{(j-2)}\triangle u^{(j-1)} + \alpha \big(u^{(j-1)}\big)^2\big]}{\big[(-\triangle)^{-1}u^{(j-1)}\big]^2}\nabla(-\triangle)^{-1}u^{(j-1)}\cdot\nabla\tilde{u}^{(j)}\\
 \end{split}\end{equation}
We estimate the first term, the second being similar.  Split it into 
 \begin{equation}\nonumber\begin{split}
&\frac{\nabla(-\triangle)^{-1}\big[(-\triangle)^{-1}u^{(j-2)}\triangle u^{(j-1)} + \alpha \big(u^{(j-1)}\big)^2\big]}{(-\triangle)^{-1}u^{(j-1)}}\cdot\nabla\tilde{u}^{(j)}\\
&=\chi_{|x|\lesssim1}\frac{\nabla(-\triangle)^{-1}\big[(-\triangle)^{-1}u^{(j-2)}\triangle u^{(j-1)} + \alpha \big(u^{(j-1)}\big)^2\big]}{(-\triangle)^{-1}u^{(j-1)}}\cdot\nabla\tilde{u}^{(j)}\\
&+\chi_{|x|\gtrsim 1}\frac{\nabla(-\triangle)^{-1}\big[(-\triangle)^{-1}u^{(j-2)}\triangle u^{(j-1)} + \alpha \big(u^{(j-1)}\big)^2\big]}{(-\triangle)^{-1}u^{(j-1)}}\cdot\nabla\tilde{u}^{(j)}\\
\end{split}\end{equation}
For the first term on the right, we get 
\begin{equation}\nonumber\begin{split}
&\|\chi_{|x|\lesssim1}\frac{\nabla(-\triangle)^{-1}\big[(-\triangle)^{-1}u^{(j-2)}\triangle u^{(j-1)} + \alpha \big(u^{(j-1)}\big)^2\big]}{(-\triangle)^{-1}u^{(j-1)}}\cdot\nabla\tilde{u}^{(j)}\|_{L^2}\\
&\lesssim \|\chi_{|x|\lesssim1}\triangle u^{(j-1)}\|_{L^4}\|u^{(j-2)}\|_{L^1\cap L^2}\|\nabla u^{(j)}\|_{L^2}+\sum_{\alpha = 0,2}\|\triangle^\alpha u^{(j-1)}\|_{L^2}\|u^{(j-1)}\|_{L^4}|\nabla u^{(j)}\|_{L^2}\\
&\lesssim t^{-\frac{3}{4}}\big[t^{\frac{1}{2}}\|\la x \ra^{-\frac{1}{2}}\nabla\triangle u^{(j-1)}\|_{L^2}+\|\triangle u^{(j-1)}\|_{L^2} +\|u^{(j-1)}\|_{L^2}\big],
\end{split}\end{equation}
where the implied constant only depends on the bounds derived in Subsection \ref{sec:control}. Furthermore, we have 
\begin{equation}\nonumber\begin{split}
&\|\chi_{|x|\gtrsim 1}\frac{\nabla(-\triangle)^{-1}\big[(-\triangle)^{-1}u^{(j-2)}\triangle u^{(j-1)} + \alpha \big(u^{(j-1)}\big)^2\big]}{(-\triangle)^{-1}u^{(j-1)}}\cdot\nabla\tilde{u}^{(j)}\|_{L^2}\\
&\lesssim \|u^{(j-2)}\|_{L^1\cap L^2}\|\triangle u^{(j-1)}\|_{L^2}\|\nabla\tilde{u}^{(j)}\|_{L^2}\\&+[\|\triangle u^{(j-1)}\|_{L^2}+\|u^{(j-1)}\|_{L^1}]\|u^{(j-1)}\|_{L^1\cap L^4}\|\nabla\tilde{u}^{(j)}\|_{L^2}\lesssim t^{-\frac{3}{4}}(\|\triangle u^{(j-1)}\|_{L^2}+1)
\end{split}\end{equation}
This completes our estimation of $\|\triangle\big(\frac{\partial_s g_j}{g_j}\tilde{u}^{j}\big)\|_{L^2}$. 
\\

{\it{Contribution of $\alpha \triangle\big[\frac{g_{j-1}^2}{g_j} e^{2s}\big(\tilde{u}^{(j-1)}(s, \cdot)\big)^2\big]$.}} Upon expanding, this results in a number of terms, and in particular the expression $\frac{g_{j-1}^2}{g_j} e^{2s}|\nabla\tilde{u}^{(j-1)}|^2(s, \cdot)$, where we omit the constant $\alpha$. Here we place both factors $\nabla\tilde{u}^{(j-1)}$ into $L^4$, taking advantage of Gagliardo-Nirenberg's inequality: 
\begin{equation}\nonumber\begin{split}
\|\la x \ra^{-\frac{1}{2}}\nabla\tilde{u}^{(j-1)}\|_{L^4}&\lesssim \|\la x \ra^{-\frac{1}{2}}\nabla\tilde{u}^{(j-1)}\|_{L^\infty}^{\frac{1}{2}}\|\la x \ra^{-\frac{1}{2}}\nabla\tilde{u}^{(j-1)}\|_{L^2}^{\frac{1}{2}}\\ 
&\lesssim [\|\la x \ra^{-\frac{1}{2}}\nabla\tilde{u}^{(j-1)}\|_{L^p} + \|\la x \ra^{-\frac{1}{2}}\nabla^2\tilde{u}^{(j-1)}\|_{L^p}]^{\frac{1}{2}}\|\la x \ra^{-\frac{1}{2}}\nabla\tilde{u}^{(j-1)}\|_{L^2}^{\frac{1}{2}}\\ 
\end{split}\end{equation}
for some $p\in (3, 6)$. Further, we have 
\[
\|\la x \ra^{-\frac{1}{2}}\nabla\tilde{u}^{(j-1)}\|_{L^p}\lesssim \|\tilde{u}^{(j-1)}\|_{L^2}^{\frac{1}{2}-}\|\triangle \tilde{u}^{(j-1)}\|_{L^2}^{\frac{1}{2}+} + \|\la x\ra^{-\frac{1}{2}}\tilde{u}^{(j-1)}\|_{L^p}
\]
\begin{equation}\nonumber\begin{split}
\|\la x \ra^{-\frac{1}{2}}\nabla^2 u^{(j-1)}\|_{L^p}&\lesssim \|\la x \ra^{-\frac{1}{2}}\nabla\triangle u^{(j-1)}\|_{L^2}^{\frac{5}{6}+}\|\la x \ra^{-\frac{1}{2}}u^{(j-1)}\|_{L^2}^{\frac{1}{6}-}\\
&+\|\triangle u^{(j-1)}\|_{L^2}^{\frac{5}{6}+}\|u^{(j-1)}\|_{L^2}^{\frac{1}{6}-}+\|u^{(j-1)}\|_{L^2}
\end{split}\end{equation}
Combining these estimates, we deduce the bound
\[
 \|\la x \ra^{-\frac{1}{2}}\nabla\tilde{u}^{(j-1)}(t, \cdot)\|_{L^4}^2\lesssim t^{-\frac{11}{12}}[\|\triangle\tilde{u}^{(j-1)}\|_{L^2} + t^{\frac{1}{2}}\|\la x \ra^{-\frac{1}{2}}\nabla\triangle\tilde{u}^{(j-1)}\|_{L^2}+1]
 \]
 The remaining terms in the expansion of $\alpha \triangle\big[\frac{g_{j-1}^2}{g_j} e^{2s}\big(\tilde{u}^{(j-1)}(s, \cdot)\big)^2\big]$ are treated like the preceding terms and omitted. 
\\

To summarize the preceding discussion, we obtain the following bound: 
 \begin{equation}\nonumber\begin{split}
&\|\int_0^t U(t,s)\triangle \big[ - \frac{\partial_s g_j}{g_j}\tilde{u}^{j}(s, \cdot) + \alpha \frac{g_{j-1}^2}{g_j} e^{2s}\big(\tilde{u}^{(j-1)}(s, \cdot)\big)^2\big]\,ds\|_{Z}\\
&\lesssim \int_0^t s^{-(1-)}[\|\triangle \tilde{u}^{(j-1)}\|_{L^2} + s^{\frac{1}{2}}\|\la x \ra^{-\frac{1}{2}}\nabla\triangle \tilde{u}^{(j-1)}\|_{L^2} +1]\,ds + T^{\nu}\|\triangle\tilde{u}^{(j)}\|_Z,
\end{split}\end{equation}
and furthermore, taking the supremum over $t\in [0, T]$, we obtain the bound (recall \eqref{eqn: Deltau} and the followig estimates)
\[
\|\triangle\tilde{u}^{(j)}\|_{Z}\lesssim T^{\nu}\big[\|\triangle\tilde{u}^{(j)}\|_{Z} +\|\triangle\tilde{u}^{(j-1)}\|_{Z}\big] + \|\triangle \tilde{u}_0\|_{L^2}+1
\]
where the implicit constant only depends on the bounds derived in Subsection \ref{sec:control}. We conclude that the bound 
\[
\|\triangle\tilde{u}^{(j-1)}\|_{Z}\leq D_4
\]
is recovered, provided $D_4$ is large enough in relation to $ \|\triangle \tilde{u}_0\|_{L^2}$ and the a priori bounds derived in Subsection \ref{sec:control}, and $T$ is small enough {\it{in relation to the a priori bounds derived in Subsection \ref{sec:control}}}. This completes the higher derivative bounds of the lemma for $|\alpha|=2$, and the ones for $|\alpha|=1$ follow by interpolation.  The proof of Lemma \ref{lem:bounds} is finally completed.

\subsection{Convergence of the $u^{(j)}$}\label{sec:convergence}
 In order to complete the proof of Proposition~\ref{prop:local}, we need to show that the iterates $u^{(j)}$ constructed in Lemma~\ref{lem:bounds} actually converge to a local-in-time solution, on some slice $[0, \tilde{T}]\times\threed$. Recall that the interval $[0,T]$ on which we proved a priori bounds on the iterates only depends on 
\[
\|u_0\|_{X}, \quad r_0,  \quad \int_{r_0<|y|<r_0^{-1}}u_0\,dy.
\]
Yet for the proposition, we may work on $[0, \tilde{T}]$ where $\tilde{T}>0$ depends in addition on $\|\triangle u_0\|_{L^2}$. Now consider \eqref{eqn:modelj} for the iterates $j$ and $j-1$.  Subtracting \eqref{eqn:modelj} for $j$ with its counterpart for $j-1$ we deduce the difference equation
\begin{equation}\nonumber\begin{split}
\partial_t\big[u^{(j)}-u^{(j-1)}\big] =& ~ (-\triangle)^{-1}u^{(j-1)}\triangle\big[u^{(j)}-u^{(j-1)}\big]
+
B^{(j)},
\\ 
& \big[u^{(j)}-u^{(j-1)}\big](0, \cdot) = 0.
\end{split}\end{equation}
Here we use the definition
$$
B^{(j)} \eqdef \triangle u^{(j-1)}\big[(-\triangle)^{-1}u^{(j-1)} - (-\triangle)^{-1}u^{(j-2)}\big]
+\alpha \big[\big(u^{(j-1)}\big)^2 - \big(u^{(j-2)}\big)^2\big].
$$
Proceeding as in the derivation of \eqref{eqn:tildemodelj}, we obtain
$$
\partial_t \dNOTj+A(t) \dNOTj = 
-\frac{\partial_t g_j}{g_j} \dNOTj
+
e^{-t} g_j^{-1}
B^{(j)},
$$
where
$
\dNOTj \eqdef e^{-t} g_j^{-1} \left(u^{(j)}-u^{(j-1)} \right)
$
and we recall the definitions from \eqref{operatorDEF}.

Now, using Duhamel, 
we obtain the integral equation 
\begin{equation}\label{eqn:jj-1}\begin{split}
\dNOTj
=
\int_0^t U(t, s)
\left\{
-\frac{\partial_s g_j}{g_j} \dNOTj + e^{-s} g_j^{-1} B^{(j)}
\right\} ~ 
ds.
\end{split}\end{equation}
We now intend to use the a priori bounds derived in Section \ref{sec:control} through Section \ref{sec:higherDb} to estimate the source terms on the right. Here the expression 
\[
e^{-t} g_j^{-1} \triangle u^{(j-1)}\big[(-\triangle)^{-1}u^{(j-1)} - (-\triangle)^{-1}u^{(j-2)}\big],
\]
appears somewhat delicate and requires us to iterate once more. 
We will use \eqref{asymptoticU} implicitly several times in the following developments.

Specifically, using \eqref{NewtonFormula3D}, we write 
\begin{equation}\nonumber\begin{split}
e^{-t} g_j^{-1}\triangle u^{(j-1) } & \left\{(-\triangle)^{-1}u^{(j-1)} - (-\triangle)^{-1}u^{(j-2)}\right\}
\\
= &  e^{-t} g_j^{-1}\triangle u^{(j-1)}\frac{1}{4\pi|x|}\int_{|y|\leq|x|}\big(u^{(j-1)}(t, y) - u^{(j-2)}(t, y)\big)\,dy
\\
& + 
e^{-t} g_j^{-1} \triangle u^{(j-1) }\int_{|y|>|x|}\frac{\big(u^{(j-1)}(t, y) - u^{(j-2)}(t, y)\big)}{4\pi|y|}\,dy.
\end{split}\end{equation}
Note that we have 
\begin{equation}\label{est:same}\begin{split}
&\|e^{-t} g_j^{-1}\triangle u^{(j-1)}\frac{1}{4\pi|x|}\int_{|y|\leq|x|}\big(u^{(j-1)}(t, y) - u^{(j-2)}(t, y)\big)\,dy\|_{L^2}
\\&\lesssim \|\ang{x}^{1/2}\triangle u^{(j-1)}\|_{L^2}\|\dNOTjO(t)\|_{L^2}.
\end{split}\end{equation}
For the second term in the expansion above we further split 
\begin{equation}\nonumber\begin{split}
&e^{-t} g_j^{-1}\triangle u^{(j-1)}\int_{|y|>|x|}\frac{\big(u^{(j-1)}(t, y) - u^{(j-2)}(t, y)\big)}{4\pi|y|}\,dy\\
= & e^{-t} g_j^{-1}\triangle u^{(j-1)}
\int_{|y|>|x|}\chi_{|y|\lesssim \la x \ra}\frac{\big(u^{(j-1)}(t, y) - u^{(j-2)}(t, y)\big)}{4\pi|y|}\,dy\\
& +e^{-t} g_j^{-1}\triangle u^{(j-1)}\int_{|y|>|x|}\chi_{|y|\gtrsim\la x \ra}\frac{\big(u^{(j-1)}(t, y) - u^{(j-2)}(t, y)\big)}{4\pi|y|}\,dy.
\\
\end{split}\end{equation}
For the first term on the right, we again have the same estimate \eqref{est:same}.
For the second integral above involving the cutoff $\chi_{|y|\gtrsim \la x \ra}$, such an estimate unfortunately fails logarithmically. Hence we go one step deeper into the iteration and replace
\begin{equation}\nonumber\begin{split}
&\int_{|y|>|x|}\chi_{|y|\gtrsim\la x \ra}\frac{\big(u^{(j-1)}(t, y) - u^{(j-2)}(t, y)\big)}{4\pi|y|}\,dy\\
&=\int_0^t ds\int_{|y|>|x|}\chi_{|y|\gtrsim\la x \ra}\frac{\triangle_{j-2}^{j-1}\big((-\triangle)^{-1}u^{(k-1)}\triangle u^{(k)} + \alpha \big(u^{(k)}\big)^2\big)}{4\pi|y|}\,dy\\
\end{split}\end{equation}
where $\Delta_{j-2}^{j-1}$ indicates the difference of the expression for $k=j-2, k=j-1$. Then using integration by parts, we get 
\begin{equation}\nonumber\begin{split}
&\int_0^t ds\int_{|y|>|x|}\chi_{|y|\gtrsim\la x \ra}\frac{\triangle_{j-2}^{j-1}\big((-\triangle)^{-1}u^{(k-1)}\triangle u^{(k)}\big)}{|y|}\,dy\\
&=-\int_0^t ds\int_{|y|>|x|}\nabla\big[\frac{\chi_{|y|\gtrsim\la x \ra}}{|y|}(-\triangle)^{-1}\big(\triangle_{j-3}^{j-2}u^{(k)}\big)\big]\nabla u^{(j-1)}\,dy\\
&-\int_0^t ds\int_{|y|>|x|}\nabla\big[\frac{\chi_{|y|\gtrsim\la x \ra}}{|y|}(-\triangle)^{-1}\big(u^{(j-2)}\big)\big]\nabla \triangle_{j-2}^{j-1}u^{(k)}\,dy\\
\end{split}\end{equation}
The first term on the right is estimated by 
\begin{equation}\nonumber\begin{split}
&\big|e^{-t}\int_0^t ds\int_{|y|>|x|}\nabla\big[\frac{\chi_{|y|\gtrsim\la x \ra}}{|y|}(-\triangle)^{-1}\big(\triangle_{j-3}^{j-2}u^{(k)}\big)\big]\nabla u^{(j-1)}\,dy\big|\\
&\lesssim \tilde{T}^{\frac{1}{2}}\big(\max_{t\in [0, \tilde{T}]}t^{\frac{1}{2}}\|\nabla u^{(j-1)}(t)\|_{L^2}\big)
\big[\|\dNOTjT(t)\|_{L^2}\\&\hspace{4cm}+\|\int_{|y|\geq|x|}\chi_{|y|\gtrsim |x|}\frac{u^{(j-2)}-u^{(j-3)}}{4\pi|y|}\,dy\|_{L_x^{\infty}}\big].
\end{split}\end{equation}
The second term above is estimated by 
\begin{equation}\nonumber\begin{split}
&e^{-t}\int_0^t ds\int_{|y|>|x|}\nabla\big[\frac{\chi_{|y|\gtrsim\la x \ra}}{|y|}(-\triangle)^{-1}\big(u^{(j-2)}\big)\big]\nabla \triangle_{j-2}^{j-1}u^{(k)}\,dy\\
&\lesssim \tilde{T}^{\frac{1}{2}}\|u^{(j-2)}\|_{L_s^\infty L_x^1}\big[\max_{t\in[0,\tilde{T}]}t^{\frac{1}{2}}\|\la x \ra^{-\frac{1}{2}}\nabla\dNOTjO(t)\|_{L^2}\big].
\end{split}\end{equation}
Combining the preceding estimates,
we easily deduce 
\begin{equation}\label{eqn: diff}\begin{split}
&\|e^{-t} {g}^{-1}_j \triangle u^{(j-1)}\big[(-\triangle)^{-1}u^{(j-1)} - (-\triangle)^{-1}u^{(j-2)}\big]\|_{L^2}
\\&+\|\int_{|y|>|x|}\chi_{|y|\gtrsim\la x \ra}\frac{\big(u^{(j-1)}(s, y) - u^{(j-2)}(s, y)\big)}{4\pi|y|}\,dy\|_{L_x^\infty}\\
&\lesssim \|\dNOTjO\|_{Z} + \|\dNOTjT\|_{Z}\\
&+ \tilde{T}^{\frac{1}{2}}\|\int_{|y|>|x|}\chi_{|y|\gtrsim\la x \ra}\frac{\big(u^{(j-2)}(s, y) - u^{(j-3)}(s, y)\big)}{4\pi|y|}\,dy\|_{L_{t,x}^\infty([0, \tilde{T}]\times\threed)}.
\end{split}\end{equation}
The remaining terms in \eqref{eqn:jj-1} are much more straightforward: 
we have 
\begin{equation}\label{eqn: 9}\begin{split}
&
\left\| e^{-t} g_j^{-1} \alpha \big[\big(u^{(j-1)}\big)^2 - \big(u^{(j-2)}\big)^2\big] \right\|_{L^2(\threed)}
\\
&
\lesssim 
\|\dNOTjO\|_{L^2}\|u^{(j-1)}+u^{(j-2)}\|_{L^\infty(\threed)}
\\&
\lesssim 
\|\dNOTjO\|_{L^2}[\| u^{(j-1)}\|_{H^2}+\| u^{(j-2)}\|_{H^2(\threed)}]
\\
&
\lesssim D_4\| \dNOTjO \|_{L^2(\threed)}.
\end{split}\end{equation}
Finally, as in {\it{(ii2)}} of Subsection \ref{sec:control}, we get 
\begin{equation}\label{eqn:10}
\|\frac{\partial_s g_j}{g_j} \dNOTj \|_{L^2}\lesssim s^{-(1-)}\big[\| \dNOTj \|_{L^2} + s^{\frac{1}{2}}\|\la x \ra^{-\frac{1}{2}}\nabla \dNOTj \|_{L^2}\big],
\end{equation}
with implied constant only depending on the bounds derived in Subsection \ref{sec:control}. 
By using \eqref{eqn: diff}, \eqref{eqn: 9}, \eqref{eqn:10} in \eqref{eqn:jj-1},
similar estimates to control the second component of $\|D^{(j)}\|_{Z}$,  and choosing $\tilde{T}$ small enough in relation to $D_1, D_2, D_3, D_4$, we deduce 
\begin{equation}\nonumber\begin{split}
&\| \dNOTj \|_{Z} + \|\int_{|y|>|x|}\chi_{|y|\gtrsim\la x \ra}\frac{\big(u^{(j-1)}(s, y) - u^{(j-2)}(s, y)\big)}{4\pi|y|}\,dy\|_{L_{s,x}^\infty([0, \tilde{T}]\times\threed)}
\\
&<\frac{1}{2}\big[\sum_{k=j-2}^{j-1}\|D^{(k)}\|_{Z}
 + \|\int_{|y|>|x|}\chi_{|y|\gtrsim\la x \ra}\frac{\big(u^{(j-2)} - u^{(j-3)}\big)}{4\pi|y|}\,dy\|_{L_{s,x}^\infty([0, \tilde{T}]\times\threed)}\big].
\end{split}\end{equation}
Here we recall $\|\cdot \|_Z$ from \eqref{Znorm}.  It follows that the $\{u^{(j)}\}_{j\geq 1}$ converge to a limit $u$ on $[0, \tilde{T}]$ satisfying the desired estimates. 

\subsection{Uniqueness}\label{sec:unq}
Let $u_{1}$ and $u_{2}$ be two solutions to \eqref{eqn:model}
with the same initial data, and 
satisfying all the properties in Proposition \ref{prop:local}. 

Then one gets the differential equation 
\begin{equation}\nonumber\begin{split}
\partial_t\big[u_1 - u_2\big] =& (-\triangle)^{-1} u_1 ~ \triangle \big[u_1 - u_2\big] + \triangle u_2 ~ \big[(-\triangle)^{-1} u_1-(-\triangle)^{-1} u_2\big]\\
&+\alpha \big[u_1^2 - u_2^2\big],\,\,\,\big[u_1 - u_2\big](0, \cdot) = 0.
\end{split}\end{equation}
But then choosing 
\[
\tilde{T} \eqdef \tilde{T}\left(\max_{i=1,2}\|u_i\|_{X}, r_0, \int_{r_0<|y|<r_0^{-1}}u_1\,dy, \max_{i=1,2}\|\triangle \tilde{u}_i\|_{L^2}\right),
\]
and replicating the immediately preceding estimates, we infer that
\[
u_1(t, \cdot) = u_2(t, \cdot), \quad \forall  t\in [0, \tilde{T}].
\]
Repeating this argument, observe that the set where $u_1$ and $u_2$ agree is open and closed and the two solutions co-incide. 
This completes the proof of Proposition \ref{prop:local}.

In the next section, we prove using some monotonicity formula that our local solutions must in fact exist globally in time.

\section{Global existence theory}\label{sec:global}

In this last section, we finally prove Theorems \ref{thm:Main} and then \ref{thm:MainImprov}.

\begin{proof}[Proof of Theorem \ref{thm:Main}] Given data $u_0$ as in the theorem, by the local existence theory we can find $T_{\text{max}}>0$ such that there exists a unique solution $u(t, \cdot)$ of \eqref{eqn:model} for $t\in [0, T_{\text{max}})$. We first show $T_{\text{max}}=\infty$. Suppose this is false. 

We immediately obtain the following monotonicity from \eqref{eqn:model}:
\[
0\leq \int_{\threed}u(t, x)\,dx\leq \int_{\threed}u_0(x)\,dx.
\]
The assumption $\alpha<\frac{1}{2}$ also implies a bound on $\|u\|_{L^{2+\delta}}$ for $\delta>0$ small enough as follows; using integration by parts we obtain
\begin{equation}\nonumber\begin{split}
\int_{\threed}\partial_t u u^{1+\delta}\,dx = &\int_{\threed}(-\triangle)^{-1} u \nabla[\nabla u u^{1+\delta}]\,dx - (1+\delta)\int_{\threed}(-\triangle)^{-1} u |\nabla u|^2 u^{\delta}\,dx\\
&+\alpha \int_{\threed} u^{3+\delta}\,dx\\
&\leq \left(\alpha-\frac{1}{2+\delta}\right)\int_{\threed} u^{3+\delta}\,dx\leq 0,
\end{split}\end{equation}
where the last inequality follows from omitting the negative term and performing an additional integration by parts for the first term using $\nabla u u^{1+\delta} = \frac{1}{2+\delta}\nabla(u^{2+\delta})$. 
We conclude that
\[
\int_{\threed}u^{2+\delta}(t, \cdot)\,dx\leq \int_{\threed}u_0^{2+\delta}\,dx,\quad t\geq 0.
\]
Next, pick $r_0>0$ such that \eqref{initialNON} holds.  The computation in Section \ref{sec:ELLcontrol} shows that we have the bound \eqref{boundUinv} holding for all $t\in [0, T_{\text{max}}]$,
where the constant $D_2 = D_2(u_0, r_0, T_{\text{max}})>0$. 

Indeed, from Section \ref{sec:ELLcontrol} we get a uniform positive lower bound on 
\[
\int_{r_0<|x|<r_0^{-1}}u(t, x)\,dx,\quad t\in [0, T_{\text{max}}].
\]
Now pick 
\[
T = T\left(\|u_0\|_{X}, r_0, \inf_{t\in [0, T_{\text{max}})}\int_{r_0<|x|<r_0^{-1}}u(t, x)\,dx\right)>0,
\]
as in Lemma~\ref{lem:bounds} and write $I = [0, T_{\text{max}}) = \bigcup_{j=1}^l I_j$ with intervals $I_j$ satisfying $|I_j| = T$.  Using the assumption $\triangle \tilde{u}_0\in L^2$ and applying Lemma~\ref{lem:bounds} successively to each $I_j$, we obtain 
\[
\sup_{t\in I}\|\triangle \tilde{u}(t, \cdot)\|_{L^2}
\leq 
C\left(\|u_0\|_{X}, ~ r_0, ~ \int_{r_0<|x|<r_0^{-1}}u_0(x)\,dx, ~T_{\text{max}}
\right)
\|\triangle \tilde{u}_0\|_{L^2}.
\]
But then Proposition \ref{prop:local} grants
$$
\tilde{T} = \tilde{T}\left(\|u_0\|_{X}+\|\triangle \tilde{u}_0\|_{L^2}, r_0, \int_{r_0<|x|<r_0^{-1}}u_0(x)\,dx, T_{\text{max}}\right)
>0,
$$
such that the solution $u(t, \cdot)$ extends 
to $[0, T_{\text{max}}+\tilde{T})$, which contradicts  $T_{\text{max}}<\infty$.

{\it{Decay at infinity}}.  Note that  for $t_1>t_2$ a solution to \eqref{eqn:model} satisfies
\[
\int_{\threed}u(t_1, \cdot)\,dx - \int_{\threed}u(t_2, \cdot)\,dx = (\alpha-1)\int_{t_1}^{t_2}\int_{\threed}u^2(s, x)\,dxds,
\]
whence we have 
$$
\lim_{T\to\infty}\int_{T}^{\infty}\int_{\threed}u^2(s, x)\,dxds = 0.
$$
This follows because we have an a priori bound on $\int u(t_i, \cdot) dx$ for $i=1,2$, and the quantity below is non-negative.  This implies that
$$
(1-\alpha)\int_{t_1}^{t_2}\int_{\R^3}u^2(s, x) dx ds
$$
is bounded uniformly with respect to $t_1$, $t_2$ and of course increasing with respect to $t_2$.  Hence the limit as $t_2\rightarrow\infty$ exists and is given by
$$
(1-\alpha)\int_{t_1}^\infty\int_{\R^3}u^2(s, x) dx ds.
$$
This function is non-increasing with respect to $t_1$ so that the assertion follows.

In particular, there exists a sequence $t_n\rightarrow\infty$ with 
\[
\int_{\threed}u^2(t_n, x)\,dx\rightarrow 0,
\]
 and by \eqref{monEST}, the $L^1$-a priori bound, and H\"{o}lder's inequality, we get 
 \[
 \| u(t_n)\|_{L^q(\threed)}\rightarrow 0,\quad 1<q\leq 2.
 \]
 But the monotonicity established above for $\|u(t,\cdot)\|_{L^q(\threed)}$, $1\leq q\leq 2$ implies 
\[
\lim_{T\to\infty}\int_{T}^{\infty}\int_{\threed}u_t u^{q-1}\,dxdt = 0.
\]
It follows that 
\[
\lim_{T\to\infty}\|u(T, \cdot)\|_{L^q(\threed)} = 0,\quad q\in (1, 2].
\]
This completes the proof of Theorem \ref{thm:Main}.
\end{proof}

Based on many of the computations in Theorem \ref{thm:Main} we will now prove Theorem \ref{thm:MainImprov} after deducing  additional a priori bounds on $\|u(t, \cdot)\|_{L^{2+}(\threed)}$ when $\alpha \in[0, 2/3)$: 

\begin{proof}[Proof of Theorem \ref{thm:MainImprov}] Let $u(t, \cdot)$ be a solution of \eqref{eqn:model}. We show that for some small $\delta>0$, and any $T>0$ such that $u(t, \cdot)$ is defined on $[0, T)\times\threed$ we have 
\[
\limsup_{t\rightarrow T}\|u(t)\|_{L^{2+\delta}(\threed)}<\infty.
\]
Once this is known, the theorem follows as in the last proof.  Using the assumption $\alpha<\frac{2}{3}$, we easily infer as in the preceding that
\[
\limsup_{t\rightarrow T}\|u(t, \cdot)\|_{L^{\frac{3}{2}+\gamma}}<\infty,
\]
for $\gamma>0$ sufficiently small.  Consider 
\begin{equation}\nonumber\begin{split}
\int_{\threed}\partial_t u u^{1+\delta}\,dx &= \int_{\threed}\big[(-\triangle)^{-1}u\triangle u +\alpha u^2\big]u^{1+\delta}\,dx\\
&=\left(\alpha - \frac{1}{2+\delta}\right)\int_{\threed}u^{3+\delta}\,dx - (1+\delta)\int_{\threed}(-\triangle)^{-1}u|\nabla u|^2 u^\delta\,dx.
\end{split}\end{equation}
Then perform an integration by parts to obtain (with $\kappa>0$ to be chosen)
\begin{align}
\int_{\threed}u^{3+\delta}\,dx &\nonumber= (2+\delta)\int_0^\infty \big(\frac{1}{r}\int_0^r u(s) s^2\,ds\big)u^{1+\delta}\partial_r u\, 4\pi r\,dr\\
&\label{eqn: I}=(2+\delta)\int_0^\infty \big(\frac{1}{r}\int_0^r u(s) s^2\,ds\big)\chi_{r\lesssim \kappa}\frac{u^{1+\frac{\delta}{2}}}{r}\partial_r u\, u^{\frac{\delta}{2}}\, 4\pi r^2\,dr\\
&\label{eqn: II}+(2+\delta)\int_0^\infty \big(\frac{1}{r}\int_0^r u(s) s^2\,ds\big)\chi_{r\gtrsim \kappa}\frac{u^{1+\frac{\delta}{2}}}{r}\partial_r u\,u^{\frac{\delta}{2}}\, 4\pi r^2\,dr.
\end{align}
To estimate \eqref{eqn: II}, we use Cauchy's inequality with $\gamma_0>0$ to write 
\begin{equation}\nonumber\begin{split}
&(2+\delta)\int_0^\infty \big(\frac{1}{r}\int_0^r u(s) s^2\,ds\big)\chi_{r\gtrsim \kappa}\frac{u^{1+\frac{\delta}{2}}}{r}\partial_r u\, u^{\frac{\delta}{2}}\, 4\pi r^2\,dr\\
&\leq (\frac{2+\delta}{2})\int_0^\infty \big(\frac{1}{r}\int_0^r u(s) s^2\,ds\big)\chi_{r\gtrsim \kappa}\big[\frac{\la r\ra}{\gamma_0}\frac{u^{2+\delta}}{r^2}+\frac{\gamma_0}{\la r\ra}|\partial_r u|^2\, u^{\delta}\big]\, 4\pi r^2\,dr\\
&\leq \gamma_0 C(\|u_0\|_{L^{\frac{3}{2}+\gamma}\cap L^1})\int_{\threed}\frac{1}{\la r\ra}|\partial_r u|^2\, u^{\delta}\,dx + C_1(\|u_0\|_{L^{\frac{3}{2}+\gamma}\cap L^1}, \gamma_0, \kappa).
\end{split}\end{equation}
To estimate \eqref{eqn: I}, first use Cauchy-Schwarz's inequality as
\begin{equation}\nonumber\begin{split}
&(2+\delta)\int_0^\infty \big(\frac{1}{r}\int_0^r u(s) s^2\,ds\big)\chi_{r\lesssim \kappa}\frac{u^{1+\frac{\delta}{2}}}{r}\partial_r u\, u^{\frac{\delta}{2}}\, 4\pi r^2\,dr\\
&\lesssim \kappa^{\nu}\big(\int_{\threed}\chi_{r\lesssim \kappa}^2\frac{u^{2+\delta}}{r^2}\,dx\big)^{\frac{1}{2}}\big(\int_{\threed}\tilde{\chi}_{r\lesssim \kappa}(\partial_r u)^2 u^\delta\,dx\big)^{\frac{1}{2}},
\end{split}\end{equation}
where we use $\frac{1}{r}\chi_{r\lesssim \kappa}\int_0^r u(s) s^2\,ds\lesssim \kappa^{\nu}$ for suitable $\nu=\frac{3}{q} - 1>0$
where $q=\frac{3+\gamma}{1+\gamma}<3$.

Further using Hardy's inequality we obtain
\begin{equation}\nonumber\begin{split}
&\kappa^{\nu}\big(\int_{\threed}\chi_{r\lesssim \kappa}^2\frac{u^{2+\delta}}{r^2}\,dx\big)^{\frac{1}{2}}\big(\int_{\threed}\tilde{\chi}_{r\lesssim \kappa}(\partial_r u)^2 u^\delta\,dx\big)^{\frac{1}{2}}\\
&\lesssim \kappa^{\nu}\big(\int_{\threed}(\chi_{r\lesssim \kappa}\partial_r u +\chi_{r\lesssim \kappa}'\frac{u}{r})^2 u^\delta \,dx\big)^{\frac{1}{2}}\big(\int_{\threed}\tilde{\chi}_{r\lesssim r_0}(\partial_r u)^2 u^\delta\,dx\big)^{\frac{1}{2}}\\
&\lesssim \kappa^\nu\big[\int_{\threed}\tilde{\chi}_{r\lesssim \kappa}(\partial_r u)^2 u^\delta\,dx
+
\int_{\threed}\big(\chi_{r\lesssim \kappa}'u\big)^2 u^\delta\, dx\big].
\end{split}\end{equation}
In the preceding, 
we have chosen the cutoff $\tilde{\chi}_{r\lesssim \kappa}$ such that $\tilde{\chi}_{r\lesssim \kappa}\chi_{r\lesssim \kappa} = \chi_{r\lesssim \kappa}$. Also, the implied absolute constant only depends on $\|u\|_{L^{\frac{3}{2}+\gamma}}$. Combining the above estimates for \eqref{eqn: I} and \eqref{eqn: II}, we infer that
\begin{equation}\nonumber\begin{split}
&\int_{\threed}\partial_t u u^{1+\delta}\,dx\\& \leq (\gamma_0+\kappa^{\nu})C_2(\|u_0\|_{L^{\frac{3}{2}+\gamma}\cap L^1})\int_{\threed}\frac{1}{\la r\ra}|\partial_r u|^2 u^\delta\,dx +C_3(\|u_0\|_{L^{\frac{3}{2}+\gamma}\cap L^1}, \gamma_0, \kappa)\\
&- (1+\delta)\int_{\threed}(-\triangle)^{-1}u|\nabla u|^2 u^\delta\,dx\leq C_3(\|u_0\|_{L^{\frac{3}{2}+\gamma}\cap L^1}, \gamma_0, \kappa),
\end{split}\end{equation}
provided we choose $\kappa$ and then $\gamma_0$ small enough such that 
\[
(\gamma_0+\kappa^{\nu})C_2(\|u_0\|_{L^{\frac{3}{2}+\gamma}\cap L^1})\leq D_2,
\]
where $D_2$ is as in Lemma~\ref{lem:bounds}, recall Remark~\ref{rem: lowerbound}. We then obtain the a priori bound 
\[
\sup_{0\leq t<T}\|u(t, \cdot)\|_{L^{2+\delta}}^{2+\delta}\leq \|u_0\|_{L^{2+\delta}}^{2+\delta} + TC_3(\|u_0\|_{L^{\frac{3}{2}+\gamma}\cap L^1}, \gamma_0, \kappa).
\]
In light of Proposition~\ref{prop:local}, the solution extends globally in time. The remaining assertions in Theorem \ref{thm:MainImprov} follow by the arguments in the proof of Theorem \ref{thm:Main}. 
\end{proof}

\end{document}